\documentclass[11pt,a4paper,twoside]{amsart}

\usepackage[T1]{fontenc}
\usepackage[sc]{mathpazo}
\usepackage{amscd}
\usepackage{amsfonts}
\usepackage{amsmath}
\usepackage{amssymb}
\usepackage{amstext}
\usepackage{amsthm}
\usepackage{enumerate}
\usepackage{fancyhdr}
\usepackage{setspace}
\usepackage[all,cmtip]{xy}
\usepackage{hyperref}

\theoremstyle{definition}
\newtheorem{definition}{Definition}[section]
\newtheorem{notation}[definition]{Notation}
\newtheorem{question}[definition]{Question}
\newtheorem{rem}[definition]{Remark}
\newtheorem{example}[definition]{Example}

\theoremstyle{plain}
\newtheorem{thm}[definition]{Theorem}
\newtheorem{lemma}[definition]{Lemma}
\newtheorem{prop}[definition]{Proposition}
\newtheorem{cor}[definition]{Corollary}
\newtheorem{result}{Theorem}[section]
\newtheorem{rescor}{Corollary}[section]

\newcommand{\nodetails}[2]{ \ifthenelse{\not \boolean{details}}{#1}{#2} }

\newcommand{\cC}{\mathcal{C}}

\newcommand{\cI}{\mathcal{I}}

\newcommand{\cO}{\mathcal{O}}
\newcommand{\cP}{\mathcal{P}}

\newcommand{\fp}{P}

\newcommand{\uM}{\underline{M}}
\newcommand{\uR}{\underline{R}}

\newcommand{\upi}{\underline{\pi}}

\newcommand{\Set}{\mathcal{S}et}

\newcommand{\OrbG}{\mathcal{O}rb_G}

\newcommand{\SpG}{\mathrm{Sp}^G}
\newcommand{\sphere}{\mathbb{S}}

\DeclareMathOperator*{\hocolim}{hocolim}

\newcommand{\Loc}[2]{#1\lbrack #2^{-1} \rbrack}

\newcommand{\AG}{A(G)}
\newcommand{\Abl}{A(-)}
\newcommand{\AGpi}{A(G)_{(\fp)}}
\newcommand{\Ablpi}{A(-)_{(\fp)}} 
\newcommand{\Ind}{\mathrm{Ind}}

\newcommand{\Res}{\mathrm{Res}}

\newcommand{\Osolv}[1]{O^\mathrm{solv}(#1)}
\newcommand{\Opi}[1]{O^\fp(#1)}

\newcommand{\IL}{\mathcal{I}_L}
\newcommand{\OL}{\mathcal{O}_L}

\def \wrt {with respect to }

\newcommand{\Q}{\mathbb{Q}}
\newcommand{\Z}{\mathbb{Z}}
\newcommand{\N}{\mathbb{N}}

\newcommand{\im}{\mathrm{im}}

\newcommand{\id}{\mathrm{id}}

\newcommand{\Hom}{\mathrm{Hom}}

\renewcommand{\O}{\mathrm{O}}

\newcommand{\redfont}[1]{{\color{red} #1}}

\newcommand\restr[2]{{
\left.\kern-\nulldelimiterspace 
#1 
\vphantom{\big|} 
\right|_{#2} 
}}

\fancypagestyle{headings}{\pagestyle{fancy}}

\fancypagestyle{headings}{
\fancyhead[LE,RO]{\thepage}
\fancyhead[LO,RE]{}
\fancyhead[CE]{Benjamin B\"ohme}
\fancyhead[CO]{Multiplicativity of the idempotent splittings of $A(G)$ and $\sphere_G$}
\fancyfoot[C]{} 

}

\usepackage{ifthen}
\newboolean{details}
\setboolean{details}{false}

\usepackage[utf8]{inputenc}
\usepackage{color}
\usepackage{geometry} 
\usepackage{setspace}
\author{Benjamin B\"{o}hme}
\address{Mathematisches Institut\\
Universit\"{a}t Bonn\\
Endenicher Allee 60\\
53115 Bonn\\
Germany}
\email{boehme@math.uni-bonn.de}
\title{Multiplicativity of the idempotent splittings of \\ the Burnside ring and the $G$-sphere spectrum}

\begin{document}
\pagestyle{headings}
\maketitle
\thispagestyle{empty}

\begin{abstract}
We provide a complete characterization of the equivariant commutative ring structures of all the factors in the idempotent splitting
of the $G$-equivariant sphere spectrum, including their Hill-Hopkins-Ravenel norms, where $G$ is any finite group. Our results
describe explicitly how these structures depend on the subgroup lattice and conjugation in $G$. Algebraically, our analysis
characterizes the multiplicative transfers on the localization of the Burnside ring of $G$ at any idempotent element, which is of
independent interest to group theorists. As an application, we obtain an explicit description of the incomplete sets of norm
functors which are present in the idempotent splitting of the equivariant stable homotopy category.
\end{abstract}
\nodetails{}{\tableofcontents}

\section{Introduction}
Let $G$ be a finite group and recall that the zeroth $G$-equivariant homotopy group $\pi_0^G(\sphere)$ of the $G$-sphere spectrum
identifies with the Burnside ring $A(G)$ \cite{segal:ESHT}.
Dress's classification \cite{dress:solvable} of the primitive idempotent elements $e_L \in A(G)$ in terms of perfect subgroups $L\leq
G$ gives rise to a splitting of $G$-spectra
\numberwithin{equation}{section}
\begin{equation} \label{splitting}
\sphere \simeq \prod_{(L) \leq G} \Loc{\sphere}{e_L}
\end{equation}
where the localization $\Loc{\sphere}{e_L}$ is the sequential homotopy colimit
\[ \hocolim ( \sphere \stackrel{e_L}{\longrightarrow} \sphere \stackrel{e_L}{\longrightarrow} \ldots ) \]
along countably infinitely many copies of (a representative of) $e_L$.
The present paper investigates the multiplicative nature of this splitting. (We warn the reader that the
induced splitting on $G$-fixed points is not the tom Dieck splitting; see Remark~\ref{rem tom Dieck}.)

The sphere is a commutative monoid in any good symmetric monoidal category of $G$-spectra
and hence admits the structure of a $G$-$E_\infty$ \emph{ring spectrum},
i.e., it comes equipped with a full set of
\emph{Hill-Hopkins-Ravenel norm maps}
\[ N_K^H \colon \bigwedge_{H/K} \Res_K \sphere \to \Res_H \sphere \]
for all $K \leq H \leq G$. These are equivariantly commutative multiplication maps which feature prominently in the solution
to the Kervaire invariant problem \cite{HHR}. The resulting norms on homotopy groups first appeared in \cite{GM:MU-modules}.
They are multiplicative transfer maps
\[ N_K^H \colon \pi_0^K(\sphere) \cong A(K) \to A(H) \cong \pi_0^H(\sphere) \]
which equip $\upi_0(\sphere) \cong A(-)$ with the structure of a \emph{Tambara functor} \cite{tambara} (and agree with the
multiplicative transfers of $A(-)$ induced by co-induction of finite $G$-sets, see Section~\ref{section:splitting}).

It is known that norm maps behave badly with respect to Bousfield localization of spectra and levelwise localization of Tambara
functors, see Example~\ref{example hill}.
Thus, it is natural to ask about the equivariant multiplicative behavior of the idempotent splitting~(\ref{splitting}).
Throughout the paper, we will decorate the norms of a localization with a tilde to distinguish them from the norms of
the original object.

\setcounter{definition}{1}
\begin{question}[Main question, homotopy-theoretic formulation] \label{main Q htpy}
For which nested subgroups $K \leq H \leq G$ does the norm map $N_K^H$ of $\sphere$ descend to a norm map
\[ \tilde{N}_K^H \colon \bigwedge_{H/K} \Res_K \Loc{\sphere}{e_L} \to \Res_H \Loc{\sphere}{e_L} \]
on the idempotent localization $\Loc{\sphere}{e_L}$, and which norms are preserved by the idempotent splitting (\ref{splitting})?
\end{question}

\begin{question}[Main question, algebraic formulation] \label{main Q alg}
For which nested subgroups $K \leq H \leq G$ does the Green ring
$\underline{\pi}_0\Loc{\sphere}{e_L} \cong \Loc{\Abl}{e_L}$ inherit a norm map $\tilde{N}_K^H$ from that of $A(-)$, and which norms
are preserved by the idempotent splitting
\[ \Abl \cong \prod_{(L) \leq G} \Loc{\Abl}{e_L}? \]
\end{question}

We now state our main results which provide an explicit and exhaustive answer to both questions.
All of our results hold locally for any collection of primes inverted. For simplicity, we only include
the integral statements in the introduction.

\subsection{Statement of algebraic results}
The following result will be restated as Theorem~\ref{local Thm A}, including the local variants.

\begin{result} \label{Thm A}
Let $L \leq G$ be a perfect subgroup and let $e_L \in A(G)$ be the corresponding primitive
idempotent given by Dress's classification of idempotents in $A(G)$ (see Theorem~\ref{Dress
idempotents}). Fix subgroups $K \leq H \leq G$. Then
the norm map $N_K^H \colon A(K) \to A(H)$ descends to a well-defined map of multiplicative monoids
\[ \tilde{N}_K^H \colon \Loc{A(K)}{e_L} \to \Loc{A(H)}{e_L} \]
if and only if the following holds:
\begin{enumerate}
  \item[($\star$)] Whenever $L' \leq H$ is conjugate in $G$ to $L$, then $L'$ is contained in $K$.
\end{enumerate}
\end{result}

Theorem \ref{Thm A} builds on previous work by Hill-Hopkins \cite{HH:EqvarMultClosure} and Blumberg-Hill
\cite[Section~5.4]{BH:ITF} which reduced the question to understanding certain division relations
between norms and restrictions of the elements $e_L \in \pi_0^G(\sphere)$, but did not make
explicit the relationship with the subgroup structure of $G$.
The proof of Theorem A is entirely algebraic
and can be found in Section \ref{subsection:proof}.

We now record some immediate consequences of Theorem A that will be restated as Corollary~\ref{local cor alg normal} and
Corollary~\ref{local cor alg trivial}, respectively.

\setcounter{rescor}{1}
\begin{rescor} \label{cor alg normal}
Let $L \leq G$ be perfect. Then $L$ is normal in $G$ if and only if the summand $\Loc{\Abl}{e_L}$ inherits all
norms of the form $\tilde{N}_K^H$ such that $K$ contains a subgroup conjugate in $G$ to $L$.
\end{rescor}

\begin{rescor} \label{cor alg trivial}
The Green ring $\Loc{A(-)}{e_L}$ inherits all norms $\tilde{N}_K^H$ for all $K \leq H$ if and only if $L=1$ is
the trivial group. In this case, the norm maps equip $\Loc{A(-)}{e_L}$ with the structure of a Tambara functor.
\end{rescor}

For an arbitrary perfect subgroup $L \leq G$, we explain how the levelwise localization $\Loc{\Abl}{e_L}$
fits into Blumberg-Hill's framework of \emph{incomplete Tambara functors} \cite{BH:ITF}, the basics of which we recall in Section
\ref{subsection:Tambara}.
For $K \leq H \leq G$, call the $H$-set $H/K$ \emph{admissible} for $e_L$ if $K \leq H$ satisfy
the condition ($\star$) of Theorem A. Call a finite $H$-set admissible if all of its orbits are admissible.
Theorem~\ref{Thm A} is complemented by the following two structural results.

\setcounter{result}{3}
\begin{result}[see Theorem~\ref{local Thm D}] \label{Thm D}
Let $L \leq G$ be a perfect subgroup and let $e_L \in \AG$ be the corresponding primitive idempotent. Then the following hold:
\begin{enumerate}[i)]
  \item The admissible sets assemble into an indexing system $\cI_L$ (in the sense of \cite[Def.~1.2]{BH:ITF}, see Section
  \ref{subsection:indsys}) such that $\Loc{\Abl}{e_L}$ is an $\cI_L$-Tambara functor under $A(-)$.
  \item In the poset of indexing systems, $\cI_L$ is maximal among the elements that satisfy i).
  \item The map $\Abl \to \Loc{\Abl}{e_L}$ is a localization at $e_L$ in the category of $\cI_L$-Tambara functors.
\end{enumerate}
\end{result}

\setcounter{rescor}{4}
\begin{rescor}[see Corollary~\ref{local cor alg splitting}] \label{cor alg splitting}
The localization maps $\Abl \to \Loc{\Abl}{e_L}$ assemble into a canonical isomorphism of $\cI$-Tambara functors 
\[ \Abl \to \prod_{(L) \leq G \; \mathrm{perfect}} \Loc{\Abl}{e_L}, \]
where $\cI$ is the intersection
\[ \cI = \bigcap_{(L) \leq G} \cI_L \]
of the indexing systems given by Theorem~\ref{Thm D}.
\end{rescor}

Together, Theorem~\ref{Thm A}, Theorem~\ref{Thm D} and Corollary~\ref{cor alg splitting} answer Question~\ref{main Q alg}. A simple
characterization of the norms parametrized by $\cI$ can be found in Lemma~\ref{admissible sets of I}.

\nodetails{}{\redfont{LATER: Does this hold more generally? Try to prove: Inverting elements in an ITF structured by an indexing system
always yields an ITF structured by some (potentially smaller) indexing systems.}}

\subsection{Statement of homotopical results}
It was conjectured by Blumberg-Hill \cite[Section~5.2]{BH:OperMult} and proven in
\cite{GW, rubin:realization, bonventre-pereira} that any indexing system can be realized by an
$N_\infty$ operad which encodes norms precisely for the admissible sets of that indexing system. In particular, for any of the indexing systems $\cI_L$ of Theorem \ref{Thm D}, we can
choose a corresponding $\Sigma$-cofibrant $N_\infty$ operad $\cO_L$. See Section \ref{subsection:poset} for details.

We use general preservation results for $N_\infty$ algebras under localization \cite{HH:EqvarMultClosure,
GW} to lift our algebraic results about $\cI_L$-Tambara functor structures on homotopy groups to a
homotopical statement about $\cO_L$-algebra structures on $G$-spectra. The following result is restated as
Corollary \ref{local cor htpy}.

\setcounter{rescor}{5}
\begin{rescor} \label{cor htpy}
For $L \leq G$ $\fp$-perfect, let $\cO_L$ be any $\Sigma$-cofibrant $N_\infty$ operad whose associated
indexing system is $\cI_L$.
Then:
\begin{enumerate}[i)]
  \item The $G$-spectrum $\Loc{\sphere}{e_L}$ is an $\cO_L$-algebra under $\sphere$.
  \item In the poset of homotopy types of $N_\infty$ operads, $\cO_L$ is maximal among the elements that satisfy i).
  \item The map $\sphere \to \Loc{\sphere}{e_L}$ is a localization at $e_L$ in the category of $\cO_L$-algebras.
\end{enumerate}
\end{rescor}

A homotopical reformulation of Corollary \ref{cor alg
trivial} shows that the idempotent splitting of $\sphere$ is far from being a splitting of $G$-$E_\infty$ ring spectra.

\setcounter{rescor}{6}
\begin{rescor}[see Corollary~\ref{local cor htpy trivial}] \label{cor htpy trivial}
The $G$-spectrum $\Loc{\sphere}{e_L}$ is a $G$-$E_\infty$ ring spectrum if and only if $L=1$ is the trivial group.
\end{rescor}

Locally at a prime $p$, this recovers a (currently unpublished) result of Grodal
\cite[Cor.~5.5]{grodal:burnside}, which we state as Theorem \ref{thm grodal}.

There is a homotopical analogue of Corollary~\ref{cor alg splitting}.

\setcounter{rescor}{7}
\begin{rescor}[see Corollary~\ref{local cor htpy splitting}] \label{cor htpy splitting}
Let $\cO$ be any $\Sigma$-cofibrant $N_\infty$ operad whose associated indexing system is $\cI$.
Then the idempotent splitting
\[ \sphere \simeq \prod_{(L) \leq G} \Loc{\sphere}{e_L} \]
is an equivalence of $\cO$-algebras, where the product is taken over conjugacy classes of perfect subgroups.
\end{rescor}

Together, Corollary~\ref{cor htpy} and Corollary~\ref{cor htpy splitting} answer Question \ref{main Q htpy}.

\subsection{Examples}
In Section~\ref{section:examples}, we use our results to explicitly calculate the multiplicative structure of the
idempotent splittings of the sphere in the case of the alternating group $A_5$ and the symmetric group $\Sigma_3$ (working
$3$-locally). Moreover, for arbitrary $G$, we deduce that the rational idempotent splitting of $\sphere_{\Q}$ cannot preserve any
non-trivial norm maps. The latter is not a new insight, cf.~e.g.~\cite[Section~7]{BGK:AlgebraicModelNaive}.

\subsection{Applications to modules}
Corollary~\ref{cor htpy}, together with the theory of modules of \cite{BH:modules-v2}, also characterizes the norm functors which
arise on the level of modules over the $N_\infty$ ring $\Loc{\sphere}{e_L}$ and its restrictions to subgroups.

The following result will be restated as Corollary~\ref{local norms modules}.
\nodetails{}{\redfont{LATER: Say a bit more here?}}

\begin{cor} \label{norms modules}
Let $L \leq G$ be perfect and let $\cO_L$ as in Corollary~\ref{cor htpy}. Assume furthermore that $\cO_L$ has the homotopy type of
the linear isometries operad on some (possibly incomplete) universe $U$.
For all admissible sets $H/K$ of $\cI_L$, there are norm functors
\[ {~}_{\Res_H(\Loc{\sphere}{e_L})} N_{K, \Res_K(U)}^{H, \Res_H(U)} \colon
Mod(\Res^G_K(\Loc{\sphere}{e_L})) \to Mod(\Res^G_H(\Loc{\sphere}{e_L})) \]
built from the smash product relative to $\Loc{\sphere_{(\fp)}}{e_L}$
which satisfy a number of relations analogous to those for the norm functor $\mathrm{Sp}^H \to \SpG$, stated
in \cite[Thm.~1.3]{BH:modules-v2}.
\end{cor}

Any $G$-spectrum is a module over $\sphere$, hence the idempotent splitting (\ref{splitting}) of $\sphere$ induces a splitting of
the category of $G$-spectra. Corollary~\ref{norms modules} then says that this does not give rise to a splitting of $G$-symmetric
monoidal categories in the sense of \cite{HH:EqvarSymMon}. Indeed, the categories of modules over (restrictions to subgroups of)
$\Loc{\sphere}{e_L}$ will only admit an incomplete set of norm functors, which then can be read off from Theorem~\ref{Thm A}.

\subsection{Topological $K$-theory spectra}
We will answer the analogues of our main questions for $G$-equivariant complex and real topological $K$-theory
in the sequel \cite{boehme:idempotent-characters_v2}, see Section~\ref{section:applications}.

\textbf{Organization:}
Section \ref{section:background} provides some background material on $N_\infty$ operads and
their algebras in $G$-spectra, (incomplete) Tambara functors, indexing systems and their behavior under localization.
In Section \ref{section:splitting}, we recall Dress's classification of idempotent elements in the Burnside ring and explain how
to obtain the splitting (\ref{splitting}) of the $G$-equivariant sphere spectrum.
We state and prove our results (including the local variants) in Section \ref{section:results} and discuss examples in Section~
\ref{section:examples}.
Finally, applications are discussed in Section \ref{section:applications}.

\textbf{Acknowledgements:} The present work was part of the author's PhD project at the
University of Copenhagen; a previous version of the article was included in his PhD thesis
\cite{boehme:thesis}.
The author would like to thank his advisor Jesper Grodal, his PhD committee consisting of Andrew Blumberg, John Greenlees
and Lars Hesselholt, as well as Markus Hausmann, Mike Hill, Joshua Hunt, David White and an anonymous referee
for many helpful discussions, suggestions and comments.
Special thanks go to Malte Leip for suggesting a formulation of the condition ($\star$) in Theorem~\ref{Thm A}
which is simpler than the one given in a previous version of this paper. \\ This research was supported by the
Danish National Research Foundation through the Centre for Symmetry and Deformation (DNRF92).
\section{Preliminaries} \label{section:background}
We briefly recall some background material on $N_\infty$ operads and $N_\infty$ ring spectra, incomplete Tambara functors and
localizations. Most of this section follows \cite{BH:OperMult, BH:ITF}.

\subsection{$N_\infty$ operads and indexing systems} \label{subsection:operads} \label{subsection:indsys}
Recall that a subgroup $\Gamma \leq G \times \Sigma_n$ is a \emph{graph subgroup} if it is the graph of a group homomorphism $H \to
\Sigma_n$ for some $H \leq G$, or equivalently, if $\Gamma \cap (\{1\} \times \Sigma_n)$ is trivial. By a $G$-\emph{operad} we mean
a symmetric operad in the category of (unbased) $G$-spaces.

\begin{definition}[\cite{BH:OperMult}, Def.~1.1]
A $G$-operad $\mathcal{O}$ is called an $N_\infty$ \emph{operad} if each $G$-space $\mathcal{O}(n)$
is a universal space for a family $\mathcal{F}_n$ of graph subgroups of $G \times \Sigma_n$ which contains
all graphs of trivial homomorphisms, i.e., all subgroups of the form $H \times \{\id\}$.
\end{definition}

The following properties are immediate from the definition.

\begin{lemma} For an $N_\infty$ operad $\mathcal{O}$, the following holds:
\begin{enumerate}[(i)]
  \item The $G$-spaces $\mathcal{O}(0)$ and $\mathcal{O}(1)$ are $G$-equivariantly contractible.
  \item The action of $\Sigma_n$ on $\mathcal{O}(n)$ is free.
  \item The underlying non-equivariant operad is always an $E_\infty$ operad.
\end{enumerate}
\end{lemma}

\begin{example}[\cite{BH:OperMult}, Lemma~3.15]
Let $U$ be a (not necessarily complete) $G$-universe, and let $\mathcal{L}(U)$ be the associated operad of linear isometric
embeddings. Then it is a $G$-operad under the conjugation action, and it is always an $N_\infty$
operad.
\end{example}

\begin{definition}
An $H$-set $X$ of cardinality $n$ is called \emph{admissible} for $\mathcal{O}$ if the graph of the corresponding action
homomorphism $H \to \Sigma_n$ is contained in $\mathcal{F}_n$.
\end{definition}

Algebras $R$ in $G$-spectra over an $N_\infty$ operad $\mathcal{O}$ are $G$-equivariant $E_\infty$ ring
spectra which in addition admit coherent equivariant multiplications given by Hill-Hopkins-Ravenel norm
maps
\cite[Thm.~6.11]{BH:OperMult}
\[ N_K^H \colon \bigwedge_{H/K} \Res^G_K(R) \to \Res^G_H(R) \]
for those nested subgroups $K \leq H \leq G$ such that $H/K$ is an admissible set for $\mathcal{O}$. (More generally, there is a
norm map $N_f$ associated to a map of $G$-sets $f \colon X \to Y$ provided that for all $y \in Y$, the preimage $f^{-1}(y)$ is an
admissible $G_y$-set, where $G_y$ denotes the stabilizer group of $y$.) Here, $\bigwedge_{H/K}$ denotes the \emph{indexed smash
product} or
\emph{Hill-Hopkins-Ravenel norm functor} \cite[Section~2.2.3]{HHR}, and the maps $N_K^H$ arise as the counits of the adjunctions \cite[Prop.~2.27]{HHR}
\begin{center}
$\xymatrix@M=10pt{
	\bigwedge_{H/K}(-) \colon \mathbf{Comm^K} \ar@<0.7ex>[r] &
	\mathbf{Comm^H} \colon \Res^H_K(-) \ar@<0.7ex>[l]
}$
\end{center}
between categories of commutative monoids in equivariant spectra.

The data of admissible $H$-sets for all $H \leq G$ can be organized in the following way: For fixed $H$, the collection of
admissible $H$-sets forms a symmetric monoidal subcategory of the category $\Set^H$ of finite $H$-sets under disjoint union.
Together, these assemble into a subfunctor $\cI$ of the coefficient system $\underline{\Set}$
whose value at $G/H$ is the symmetric monoidal category $\Set^H$. The operad structure of $\mathcal{O}$ forces $\cI$ to be closed
under certain operations, as captured in the following definition.

\begin{definition}[\cite{BH:ITF}, Def.~1.2] \label{def indexing system}
An \emph{indexing system} is a contravariant functor
\[ \cI \colon \OrbG^{op} \to Sym, \; G/H \mapsto \cI(H) \]
from the orbit category of $G$ to the category of symmetric monoidal categories and strong symmetric monoidal functors, such that
the following holds:
\begin{enumerate}[(i)]
  \item The value $\cI(H)$ of $\cI$ at $G/H$ is a full symmetric monoidal subcategory of the category $\Set^H$ of
  finite $H$-sets and $H$-equivariant maps which contains all trivial $H$-sets.
  \item Each $\cI(H)$ is closed under finite limits.
  \item The functor $\cI$ is closed under ``self-induction'': If $H/K \in \cI(H)$ and $T \in \cI(K)$,
  we require that $\Ind_K^H(T) = H \times_K T \in \cI(H)$.
\end{enumerate}
\end{definition}

The collection of all indexing systems (for a fixed group $G$) forms a poset under inclusion. $N_\infty$ operads give rise to
indexing systems.

\begin{definition}
Let $\cI$ be an indexing system. Call an $H$-set $X$ \emph{admissible} if $X \in \cI(H)$. Call a map $f \colon Y \to Z$ of finite
$G$-sets \emph{admissible} if the orbit $G_{f(y)}/G_y$ obtained from stabilizer subgroups is admissible for all $y \in Y$.
\end{definition}

\begin{prop}[\cite{BH:OperMult}, Thm.~4.17]
The admissible sets of any $N_\infty$ operad $\mathcal{O}$ form an indexing system.
\end{prop}

\subsection{The poset of $N_\infty$ ring structures} \label{subsection:poset}
Two extreme cases of $N_\infty$ operads arise:

\begin{definition}[\cite{BH:OperMult}, Section~3.1] \label{def G-E infty}
If for all $n \in \N$, $\mathcal{F}_n$ is the family of all graph subgroups of $G \times
\Sigma_n$, then $\mathcal{O}$ is called a $G$-$E_\infty$ \emph{operad} or \emph{complete}
$N_\infty$ \emph{operad}.
If for all $n$, $\mathcal{F}_n$ is the family of trivial graphs $H \times \{\id\}$, then
$\mathcal{O}$ is called a \emph{naive} $N_\infty$ \emph{operad}.
\end{definition}

Algebras over $G$-$E_\infty$ operads are equivariant $E_\infty$ ring spectra equipped with a complete
collection of norm maps and form a category which is Quillen equivalent to that of strict commutative
monoids in $G$-spectra. Naive $N_\infty$ operads are non-equivariant $E_\infty$ operads equipped with the
trivial $G$-action. Their algebras are all $G$-spectra that are underlying $E_\infty$
ring spectra with no specified non-trivial norm maps. The $N_\infty$ operads
with other collections of admissible sets interpolate between those two extremes. We refer to
\cite[Section~6]{BH:OperMult} for proofs and further details.

The collection of homotopy classes of $N_\infty$ operads forms a poset that only depends on the
combinatorial data of the admissible sets, as we recall now.

\begin{definition}[\cite{BH:OperMult}, Def.~3.9]
A morphism of $N_\infty$ operads $\cO \to \cO'$ is a \emph{weak equivalence} if it induces a weak equivalence of spaces
$\cO(n)^\Gamma \to \cO'(n)^\Gamma$ for all $n \geq 0$ and all subgroups $\Gamma \leq G \times \Sigma_n$.
\end{definition}

Blumberg-Hill conjectured the following equivalence of categories and proved the ``fully faithful" part
\cite[Thm.~3.24]{BH:OperMult}.
Different proofs for the essential surjectivity were given by Gutierrez-White \cite[Thm.~4.7]{GW}, Rubin
\cite[Thm.~3.3]{rubin:realization}  and Bonventre-Pereira \cite[Cor.~IV]{bonventre-pereira}, and it should be possible to extract
an $\infty$-categorical proof from \cite{barwick:PHCT-intro} and its sequels.

\begin{thm}[Blumberg-Hill et al.]
\label{realization}
The functor from the homotopy category of $N_\infty$ operads (\wrt the above notion of weak equivalence) to the poset of indexing
systems which assigns to each $N_\infty$ operad its collection of admissible sets is an equivalence of categories.
\end{thm}

\begin{rem} \label{rem cofibrancy one}
We record a technical detail for later reference: \cite[Thm.~4.10]{GW} guarantees that for each indexing
system $\cI$, we can find a corresponding $N_\infty$ operad $\cO$ which is $\Sigma$-\emph{cofibrant}, i.e., each $\cO(n)$ has the homotopy type
of a (necessarily $\Sigma_n$-free) $(G \times \Sigma_n)$-CW complex.
This will be used in Section \ref{subsection:preservation htpy}.
\end{rem}

\subsection{Mackey functors, Green rings and (incomplete) Tambara functors}
\label{subsection:Tambara}
Recall that a \emph{Mackey functor} $\uM$ (\wrt an ambient group $G$ which we leave implicit in the notation) consists of an abelian
group $\uM(T)$ for each finite $G$-set, equipped with a structure map $\uM(X) \to \uM(Z)$ for
each span
\[ X \stackrel{r}{\longleftarrow} Y \stackrel{t}{\longrightarrow} Z, \]
subject to a list of axioms. In particular, $\uM$ is additive in the sense that $\uM(S \sqcup T) \cong \uM(S) \times \uM(T)$.
Thus, it is determined on objects by the values $\uM(H) := \uM(G/H)$ for subgroups $H \leq G$.
We refer to \cite[Section~3]{strickland:tambara} for details.

A Mackey functor $\uR$ is a \emph{Green ring} if $\uR(X)$ is a commutative ring for all $G$-sets
$X$ such that all restriction maps are ring homomorphisms and all transfers are homomorphisms of
modules over the target.

Many naturally occuring examples of Green rings such as the Burnside ring $A(-)$ or the complex representation ring $RU(-)$ come
equipped with additional multiplicative transfers, called \emph{norms}. Green rings with compatible norms are known as
\emph{Tambara functors} (originally defined as ``TNR functors" \cite{tambara}) and were generalized in \cite{BH:ITF} to cases where
only some of the norm maps are available.
We quickly review these \emph{incomplete Tambara functors}.

Let ${bispan}^G$ denote the category of \emph{bispans of $G$-sets}.
It has objects the finite $G$-sets and morphisms the isomorphism classes of bispans of finite $G$-sets
\[ X \stackrel{r}{\longleftarrow} Y \stackrel{n}{\longrightarrow} Z \stackrel{t}{\longrightarrow} W. \]
We refer to \cite[Section 6]{strickland:tambara} for the definition of composition and further details. Blumberg-Hill observed that
one can restrict the class of maps $n$ which are allowed at the central position of a bispan to encode Tambara functors with
incomplete collections of norms, as we recall now.

\begin{definition}[\cite{BH:ITF}, Sections~2.2,~3.1] \label{def D}
A subcategory $D$ of $\Set^G$ is called
\begin{enumerate}[1)]
  \item \emph{wide} if it contains all objects,
  \item \emph{pullback-stable} if any base-change of a map in $D$ is again in $D$, and
  \item \emph{finite coproduct-complete} if it has all finite coproducts and they are created
  in $\Set^G$.
\end{enumerate}
\end{definition}

\begin{thm}[\cite{BH:ITF}, Thm.~2.10] \label{restricted polynomials}
Let $D \subseteq \Set^G$ be a wide, pullback-stable subcategory, then the wide subgraph
${bispan}_D^G$ of the category of bispans that only contains morphisms of
the form
\[ X \leftarrow Y \to Z \to W \]
where $Y \to Z$ is in $D$, forms a subcategory.
\end{thm}

\begin{definition}[\cite{BH:ITF}, Def.~3.9]
For an indexing system $\cI$, let $\Set_\cI^G \subseteq \Set^G$ be the wide subgraph which contains a morphism 
$f \colon X \to Y$ if and only if for all $y \in Y$, the quotient of stabilizers $G_{f(y)}/G_y$ is in $\cI(G_{f(y)})$.
\end{definition}

\begin{thm}[\cite{BH:ITF}, Thm.~3.18] \label{iso IS D}
The assignment $\cI \mapsto \Set_\cI^G$ gives rise to an isomorphism between the poset of indexing systems and the poset of wide,
pullback-stable, finite coproduct-complete subcategories $D \subseteq \Set^G$.
\end{thm}

\begin{definition}[\cite{BH:ITF}, Def.~4.1] \label{def ITF}
Let $D \subseteq \Set^G$ be a wide, pullback-stable symmetric monoidal subcategory.
\begin{enumerate}[1)]
  \item A $D$-\emph{semi Tambara functor} is a product-preserving functor ${bispan}_D^G \to \Set$.
  \item A $D$-\emph{Tambara functor} is a $D$-Tambara functor that is abelian group valued on
  objects.
  \item For an indexing system $\cI$ and $D = \Set_{\cI}^G$, define an $\cI$-\emph{Tambara functor} to be a $D$-Tambara functor.
  \item If $D = \Set^G$, then $D$-Tambara functors are simply called Tambara functors.
\end{enumerate}
\end{definition}

\begin{rem}
We did not require that $D$ be finite coproduct-complete in Definition~\ref{def ITF}. If this also holds, i.e., if
$D$ corresponds to an indexing system $\cI$, then it can be shown that every $\cI$-Tambara functor has an
underlying Green ring and all norm maps are maps of multiplicative monoids, see \cite[Prop.~4.6 and Cor.~4.8]{BH:ITF}.
\end{rem}

The underlying object of a product in the category of bispans is the coproduct of $G$-sets, see
\cite[Prop.~7.5~(i)]{tambara}, so the condition that any $D$-Tambara functor $\uR$ be product-preserving means that
\[ \uR(S \sqcup T) \cong \uR(S) \times \uR(T) \]
for all finite $G$-sets $S$ and $T$. Hence, on the level of objects, $\uR$ is determined by the
groups $\uR(H) := \uR(G/H)$ for all $H \leq G$.

\begin{notation}
We will use the following special cases of the
structure maps frequently in the present paper:
Spans of the form
$(Y \stackrel{f}{\longleftarrow} X \stackrel{\id}{\longrightarrow} X \stackrel{\id}{\longrightarrow} X)$
give rise to \emph{restrictions} $R_f \colon \uR(Y) \to \uR(X)$ and spans of the form
$(X \stackrel{\id}{\longleftarrow} X \stackrel{\id}{\longrightarrow} X \stackrel{f}{\longrightarrow} Y)$
induce \emph{transfers} $T_f \colon \uR(X) \to \uR(Y)$. Moreover, spans of the form
$(X \stackrel{\id}{\longleftarrow} X \stackrel{f}{\longrightarrow} Y \stackrel{\id}{\longrightarrow} Y)$
give rise to \emph{norms} $N_f \colon \uR(X) \to \uR(Y)$.
If $f \colon X \to Y$ is the canonical surjection $G/K \to G/H$ arising from nested subgroup inclusions $K \leq H \leq G$, then we
write $R^H_K := R_f$, $T_K^H := T_f$ and $N_K^H := N_f$, respectively.
\end{notation}

\begin{example} \label{example TF}
The Burnside ring $A(G)$ is a Tambara functor. Restrictions, transfers and norms are given by restriction, induction and
co-induction of $G$-sets, respectively. Similarly, the complex representation ring $RU(G)$ is a Tambara functor with restrictions,
transfers and norms given by restriction, induction and tensor induction of $G$-representations, respectively. The ``linearization
map'' $A(G) \to RU(G)$ that sends a finite $G$-set to its associated permutation representation is a map of Tambara functors, i.e.,
it is compatible with all of the structure maps.
\end{example}

\nodetails{}{ \begin{example}
\redfont{LATER: Group cohomology?}
\end{example}}

Another class of examples arises from equivariant stable homotopy theory: The norm maps $N_K^H$ of an $N_\infty$ ring spectrum $R$
give rise to multiplicative transfers on equivariant homotopy groups
\[ N_K^H \colon \pi_V^K(R) \to \pi_{\Ind_K^H(V)}^H(\bigwedge_{H/K} R) \]
given by sending the $K$-equivariant homotopy class of $f
\colon \sphere^V \to R$ to the $H$-equivariant homotopy class of the composite
\[ \sphere^{\Ind_K^H(V)} \cong \bigwedge_{H/K} \sphere^V \stackrel{\bigwedge f}{\longrightarrow} \bigwedge_{H/K} R
\stackrel{N_K^H}{\longrightarrow} R, \]
see \cite[Section~2.3.3]{HHR}.

\begin{thm}[\cite{brun:eqvar-spectra}, \cite{BH:ITF}, Thm.~4.14] \label{Brun relative}
Let $R$ be an algebra over an $N_\infty$ operad $\cO$, then $\underline{\pi}_0(R)$ is an $\cI$-Tambara functor structured by the
indexing system $\cI$ corresponding to $\cO$ under the equivalence of categories from Theorem \ref{realization}.
\end{thm}

The structure on the entire homotopy ring $\{ \pi_V^H(R) \}_{H \leq G, \; V \in RO(H)}$ is described in
\cite{angeltveit-bohmann}.
For the purpose of the present paper, it suffices to consider the zeroth equivariant homotopy groups.

\subsection{Localization and $N_\infty$ rings} \label{subsection:preservation htpy}
We record some preservation results for algebraic structure under localizations of $G$-spectra
which invert a single element $x \in \pi_\ast^G(\sphere)$. For definiteness, we work in the
category of orthogonal $G$-spectra equipped with the positive complete model structure
\cite[Thm.~B.63]{HHR}.

We begin with an example that illustrates that localization at a single homotopy element need not preserve
any of the (non-trivial) norm maps of an $N_\infty$ ring spectrum.

\begin{example}[\cite{HH:EqvarSymMon}, Prop.~6.1] \label{example hill}
The inclusion of $0$ into the reduced regular representation $\tilde{\rho}$ of $G$ defines an
essential map $S^0 \to S^{\tilde{\rho}}$ of $G$-spaces all of whose restrictions to proper
subgroups are equivariantly null because they necessarily have fixed points along which
the two points of $S^0$ can be connected by an equivariant path. The resulting map gives
rise to an element $\alpha \in \pi_{-\tilde{\rho}}^G(\sphere)$ such that the resulting
$G$-spectrum $\Loc{\sphere}{\alpha}$ is non-trivial but all of its restrictions to proper
subgroups are equivariantly contractible. Thus, it cannot admit any norms
\[ \bigwedge_{G/H} \Res^G_H \Loc{\sphere}{\alpha} \to \Loc{\sphere}{\alpha} \]
because on homotopy rings, they would induce ring maps from zero rings to non-trivial rings.
\end{example}

\begin{rem}
The element $\alpha$ is not an element of the $\Z$-graded homotopy groups $\pi_\ast^G(\sphere)$, but
only of the $RO(G)$-graded homotopy groups $\pi_\star^G(\sphere)$. However, our results in Section
\ref{section:results} show that even when we restrict attention to elements $x \in \pi_0^G(\sphere)$, we can construct many other examples of the loss of $N_\infty$
structure under localization in terms of elementary group theory. Indeed, the $A_5$-spectra
$\Loc{\sphere}{e_{A_5}}$ and $\Loc{\sphere}{\alpha}$ are very similar in terms of their equivariant
multiplicative behavior, see Section~\ref{section:examples}.
\end{rem}

Idempotent homotopy elements necessarily live in degree zero. We now focus on localizations given by such
elements.

\begin{notation}
By abuse of notation, let $x$ be a map representing the homotopy class $x \in
\pi_0^G(\sphere)$. We write $\cC_{x}$ for the set of morphisms of $G$-spectra
\[ \cC_{x} = \{ G_+ \wedge_H S^n \wedge x \, | \, H \leq G, \, n \in \Z \}. \]
\end{notation}

\begin{prop}
Bousfield localization at $\cC_{x}$ has the following properties:
\begin{enumerate}[i)]
    \item It is given by smashing with $\Loc{\sphere}{x}$, hence recovers (orbitwise)
    $x$-localization on the level of equivariant homotopy groups.
    \item It is a \emph{monoidal} localization in the sense that the resulting local model
    structure is again a monoidal model category.
\end{enumerate}
\end{prop}

\begin{proof}
Since the map $ho(\SpG)(G/H_+ \wedge \sphere^n \wedge x, X)$ is just the action of $x$ on $\pi_{n}^H(X)$, we see that an
object $X$ is $\cC_x$-local if and only if its equivariant homotopy groups are $x$-local. But $x$-localization is
given by smashing with $\Loc{\sphere}{x}$. \\
For part ii), let $f$ be in $\cC_x$ and let $K$ be any cofibrant $G$-spectrum. In particular, $f$ is a
$\cC_x$-local equivalence and hence $f \wedge \Loc{\sphere}{x}$ is a weak equivalence. Smashing with cofibrant
$G$-spectra preserves equivalences, so $f \wedge K \wedge \Loc{\sphere}{x}$ is an equivalence. But this means
that $f \wedge K$ is a $\cC_x$-local equivalence. Now \cite[Thm.~4.6]{white:monoidal_v2} implies that
$\cC_x$-localization is monoidal.
\end{proof}

We now present a preservation result for $N_\infty$ ring structures due to Gutierrez
and White. A similar result first appeared in the special case of $G$-$E_\infty$ rings in
\cite[Cor.~4.11]{HH:EqvarMultClosure}; the non-equivariant version of the statement goes back at least to
\cite[Thm.~VIII.2.2]{EKMM}.

\begin{definition}[\cite{GW}, Def.~7.3] \label{def preservation}
Let $\cP$ be a $G$-operad.
A Bousfield localization $L_{\cC}$ is said to \emph{preserve} $\cP$-algebras if the following two conditions
hold:
\begin{enumerate}[(1)]
    \item If $E$ is a $\cP$-algebra, then there is some $\cP$-algebra $\tilde{E}$ which is weakly
    equivalent as a $G$-spectrum to $L_{\cC}(E)$.
    \item If $E$ is cofibrant as a $\cP$-algebra (in the model structure with equivalences and fibrations
    detected by the forgetful functor to $G$-spectra), then there is a $\cP$-algebra $\tilde{E}$, a
    $\cP$-algebra homomorphism $r_E \colon E \to \tilde{E}$ and a weak equivalence $\beta_E \colon
    L_{\cC}(E) \to \tilde{E}$ such that the underlying $G$-spectrum of $\tilde{E}$ is $\cC$-local and such
    that in the homotopy category of $G$-spectra, $r_E$ equals the composite
    \[ E \to L_{\cC}(E) \stackrel{\beta_E}{\longrightarrow} \tilde{E}, \]
    where the first map is the canonical localization map.
\end{enumerate}
\end{definition}

Recall that a $G$-operad $\cP$ is called $\Sigma$-\emph{cofibrant} if all of its spaces $\cP(n)$ have the homotopy type of $(G
\times \Sigma_n)$-CW complexes. The following preservation result is a direct translation of
\cite[Cor.~7.10]{GW} to the positive complete model structure on orthogonal $G$-spectra.

\begin{thm}[\cite{GW}, Cor.~7.10] \label{thm htpy preservation}
Let $\cP$ be a $\Sigma$-cofibrant $N_\infty$ operad. Let $L_{\cC}$ be a monoidal left Bousfield localization. Then $L_{\cC}$
preserves $\cP$-algebras in $G$-spectra if and only if the functors
\[ G_+ \wedge_H \bigwedge_{H/K} \Res^G_K (-) \colon \SpG \to \SpG \]
preserve $\cC$-local equivalences between cofibrant objects for all $H \leq G$ and all transitive
$H$-sets $H/K$ which are admissible for $\cP$.
\end{thm}

\begin{rem}
The statement of \cite[Cor.~7.10]{GW} is actually phrased in terms of the functors $G_+ \wedge_H
\bigwedge_{T} \Res^G_K (-)$ for all $H \leq G$ and all admissible $H$-sets $T$. The formulations are easily
seen to be equivalent, using that $\bigwedge_{T_1 \coprod T_2} (-) \simeq \bigwedge_{T_1} (-) \wedge
\bigwedge_{T_2} (-)$ and that the smash product of two equivalences between cofibrant objects is an
equivalence.
\end{rem}

\begin{cor}[\cite{GW}, Cor.~7.5] \label{cor naive preservation}
Any monoidal left Bousfield localization $L_{\cC}$ preserves algebras over naive $N_\infty$ operads.
Since any algebra over an arbitrary $N_\infty$ operad $\cP$ admits the structure of a naive $N_\infty$ algebra
by restricting the operad action along (a representative of) the unique map from the naive $N_\infty$ operad
to $\cP$ in the homotopy category of $N_\infty$ operads, it follows that any monoidal left Bousfield
localization sends $\cP$-algebras to naive $N_\infty$ algebras.
\end{cor}

If the localization is given by inverting a single element $x \in \pi_0^G(\sphere)$, the
condition in Theorem \ref{thm htpy preservation} can be verified on homotopy groups, as we explain now. The
following results generalize \cite[Thm.~4.11]{HH:EqvarMultClosure} to the $N_\infty$ setting in the case of the ($p$-local) sphere
spectrum.

\begin{lemma} \label{lemma norm commutes with localization}
Fix a transitive $H$-set $H/K$ and an element $x \in \pi_0^G(\sphere)$. Then:
\begin{enumerate}[i)]
  \item The functor $\bigwedge_{H/K} \Res^G_K$ admits a left derived functor $\mathbb{L}(\bigwedge_{H/K}
  \Res^G_K)$ which commutes with sifted (hence with sequential) homotopy colimits.
  \item If $\sphere_c \to \sphere$ is a cofibrant replacement in the positive complete model structure on
  orthogonal $G$-spectra (see \cite[Prop.~B.63]{HHR}), then so is $\Loc{\sphere_c}{x} \to \Loc{\sphere}{x}$.
  Moreover, $\bigwedge_{H/K} \Res^G_K \sphere_c \to \bigwedge_{H/K} \Res^G_K \sphere$ and $\bigwedge_{H/K} \Res^G_K \Loc{\sphere_c}{x}
  \to \bigwedge_{H/K} \Res^G_K \Loc{\sphere}{x}$ are cofibrant replacements in orthogonal $H$-spectra.
  \item The map
  \[ \Res^G_H(\sphere) \simeq \bigwedge_{H/K} \Res^G_K (\sphere) \to \bigwedge_{H/K} \Res^G_K
  (\Loc{\sphere}{x})
  \]
  induced from the canonical map $\sphere \to \Loc{\sphere}{x}$ induces an equivalence
  \[ \Loc{\Res^G_H(\sphere)}{(N_K^H R^G_K (x))} \to \bigwedge_{H/K} \Res^G_K (\Loc{\sphere}{x}). \]
\end{enumerate}
\end{lemma}

\begin{proof}
\emph{i):}
This is well-known for the restriction functor; for the norm it follows from \cite[Prop.~B.104]{HHR} combined
with \cite[Prop.~A.27, A.53]{HHR}.
\\
\emph{ii):}
Localization, restriction and norm preserve cofibrancy of equivariant spectra, as is well-known for the first
two of these functors and included in \cite[Prop.~B.89]{HHR} for the norm. In order to prove the three
statements that comprise ii), it now suffices to show that the three maps in question are equivalences. This
is easy to see for the map $\Loc{\sphere_c}{x} \to \Loc{\sphere}{x}$, so the first statement is clear. The
sphere is the initial commutative monoid, hence cofibrant as a commutative monoid, and so the second statement follows from \cite[Prop.~B.146]{HHR} applied
to the map $\Res^G_K(\sphere_c \to \sphere)$. The third statement then follows from the last part of
\cite[Lemma~B.151]{HHR}, using that $\Loc{\sphere}{x}$ is a filtered colimit of copies of~$\sphere$.
\\
\emph{iii):}
Consider the following diagram:
\begin{center}
$ \xymatrix@M=8pt@R=1.5pc{
	\hocolim ( \displaystyle \bigwedge_{H/K} \Res^G_K(\sphere_c) \ar[r]^{\bigwedge
	\Res(\cdot x)} \ar[d] &
	\displaystyle \bigwedge_{H/K} \Res^G_K(\sphere_c) \to \ldots ) \ar[r] &
	(\displaystyle \bigwedge_{H/K} \Res^G_K)(\Loc{\sphere_c}{x}) \ar[d] \\
	\hocolim (\displaystyle \bigwedge_{H/K} \Res^G_K(\sphere) \ar[r]^{\bigwedge
	\Res(\cdot x)} &
	\bigwedge_{H/K} \Res^G_K(\sphere) \to \ldots ) \ar@{-->}[r] &
	(\displaystyle \bigwedge_{H/K} \Res^G_K)(\Loc{\sphere}{x})
}$
\end{center}
The vertical maps are equivalences by part ii). The dashed horizontal map is induced by the map 
  \[ \Res^G_H(\sphere) \cong \bigwedge_{H/K} \Res^G_K (\sphere) \to \bigwedge_{H/K} \Res^G_K
  (\Loc{\sphere}{x}), \]
and similar for the solid horizontal one. The solid horizontal map is an equivalence by part i) and the fact
that left derived functors can be computed by passing to cofibrant replacements. Hence the dashed arrow is an
equivalence. It now suffices to see that the domain of the dashed arrow computes the localization
$\Loc{\Res^G_H(\sphere)}{(N_K^H R^G_K(x))}$. This holds because $N_K^H R^G_K(x)$ is given as the composite of
$\bigwedge_{H/K} \Res^G_K(x)$ with the norm map $\bigwedge_{H/K} \Res^G_K \sphere \to \Res^G_H(\sphere)$, and
the latter is an equivalence since the sphere is the monoidal unit.
\end{proof}

\begin{prop} \label{prop translation White HH integral}
Let $\cP$ be a $\Sigma$-cofibrant (see Remark~\ref{rem cofibrancy one}) $N_\infty$ operad and
fix $x \in \pi_0^G(\sphere)$.
Then $L_{\cC_x}$ preserves $\cP$-algebras in $G$-spectra if and only if for all $H \leq G$ and all
transitive admissible $H$-sets $H/K$, the element
$N_K^H R^G_K(x)$
divides a power of $R^G_H(x)$ in the ring $\pi_0^H(\sphere)$.
\end{prop}

\begin{proof}
We have to show that for admissible such $H/K$, the functors $G_+ \wedge_H \bigwedge_{H/K} \Res^G_K (-)$
preserve $\cC_x$-local equivalences between cofibrant objects if and only if the elements
$N_K^H R^G_K(x)$ divide powers of $R^G_H(x)$. \\
If $\cC_x$-local equivalences are preserved, then in particular the map of $G$-spectra
\[ G_+ \wedge_H \bigwedge_{H/K} \Res^G_K (x) \colon
G_+ \wedge_H \bigwedge_{H/K} \Res^G_K (\sphere) \to G_+ \wedge_H \bigwedge_{H/K} \Res^G_K (\sphere)\]
is an $x$-local equivalence. Under the standard isomorphism $\pi_\ast^G(G_+ \wedge_H -) \cong \pi_\ast^H(-)$,
the induced map on $\pi_\ast^G(-)$ agrees with multiplication by the element $N_K^H R^G_K (x)$ and becomes a unit
after inverting $R^G_H(x)$, hence the element $N_K^H R^G_K (x)$ must divide a power of $R^G_H(x)$.
\\
Conversely, assume the division relation holds and let $f \colon X \to Y$ be a $\cC_x$-local equivalence between cofibrant objects.
We have to show that $G_+ \wedge_H \bigwedge_{H/K} \Res^G_K(f)$ is a $\cC_x$-local equivalence.
Since induction is a left Quillen functor, it suffices to show that the map $\bigwedge_{H/K} \Res^G_K (f)$
becomes an equivalence of $H$-spectra upon smashing with
$\Loc{\sphere}{R^G_H(x)}$.
We are going to show that it is an equivalence upon smashing with
$\Loc{\sphere}{N_K^H R^G_K (x)}$. Since the element $N_K^H R^G_K (x)$ divides a power of $R^G_H(x)$ by
assumption, the claim then follows. \\
The map $f \wedge \Loc{\sphere}{x}$ is an equivalence by assumption, so for any cofibrant replacement
$\sphere_c \to \sphere$, the map $f \wedge \Loc{\sphere_c}{x}$ is an equivalence between cofibrant
$G$-spectra. Then $\bigwedge_{H/K} \Res^G_K(f \wedge \Loc{\sphere_c}{x})$ is an equivalence of $H$-spectra by
\cite[Prop.~B.103]{HHR}. By part ii) of Lemma~\ref{lemma norm commutes with localization}, the map
$\bigwedge_{H/K} \Res^G_K(f \wedge \Loc{\sphere}{x})$ must be an equivalence. But the norm and restriction
functors commute with smash products, so
\[ (\bigwedge_{H/K} \Res^G_K)(f) \wedge (\bigwedge_{H/K} \Res^G_K)(\Loc{\sphere}{x}) \]
is an equivalence. Finally, part iii) of Lemma~\ref{lemma norm commutes with localization} implies that
\[ (\bigwedge_{H/K} \Res^G_K)(f) \wedge \Loc{\sphere}{(N_K^H R^G_K(x))} \]
is an equivalence, which finishes the proof.
\end{proof}

\begin{cor}[\cite{HH:EqvarMultClosure}, §4; \cite{bachmann-hoyois_v4}, Lemma~12.9] \label{cor inverting primes
htpy} Let $n \in \Z$, viewed as the element $n \cdot \left[ G/G \right] \in A(G)$. Then $\sphere \lbrack \frac{1}{n} \rbrack$ is a
complete $G$-$E_\infty$ ring spectrum.
\end{cor}

Consequently, for any collection $\fp$ of primes, $\sphere_{(\fp)} :=
\sphere \lbrack q^{-1}, \, q \notin \fp \rbrack$ is a complete $G$-$E_\infty$ ring spectrum, or
equivalently, a commutative monoid in $\SpG$.
One can now mimick the proof of Proposition~\ref{prop translation White HH integral} in the
$\fp$-local case.

\begin{prop} \label{prop translation White HH}
Let $\cP$ be a $\Sigma$-cofibrant $N_\infty$ operad. Fix $x \in \pi_0^G(\sphere_{(\fp)})$. Then
$L_{\cC_x}$ preserves $\cP$-algebras in $\fp$-local $G$-spectra if and only for all $H \leq G$ and
all transitive admissible $H$-sets $H/K$, the element
$N_K^H R^G_K (x)$
divides a power of $R^G_H(x)$ in the
ring $\pi_0^H(\sphere_{(\fp)})$.
\end{prop}

\subsection{Localization and incomplete Tambara functors} \label{subsection:preservation alg}
There are analogous preservation results for incomplete Tambara functors under localization. Given an
$\cI$-Tambara functor $\uR$ and an element $x \in \uR(G)$, consider the levelwise localization
$\Loc{\uR}{x}(H) := \Loc{\uR(H)}{R^G_H(x)}$. By \cite[Lemma~10.2]{strickland:tambara}, this agrees with the
sequential colimit along countably infinitely many copies of multiplication by $x$, taken in the category of
Mackey functors. Multiplication by $x$ is typically not a map of Tambara functors, and the levelwise
localization is usually not a Tambara functor. An alternative notion of localization which enjoys a universal
property in the category of Tambara functors is discussed in \cite[Section~5.4]{BH:ITF}. The two notions
agree if and only if the division relations from Proposition~\ref{prop translation White HH integral} are satisfied.

\begin{thm}[\cite{BH:ITF}, Thm.~5.26] \label{thm alg preservation}
Let $\uR$ be an $\cI$-Tambara functor structured by an indexing system $\cI$.
Let $x \in \uR(G)$. Then the orbit-wise localization $\Loc{\uR}{x}$ is a localization in the category of $\cI$-Tambara functors if
and only if for all admissible sets $H/K$ of $\cI$, the element $N^H_K R^G_K(x)$ divides a power
of $R^G_H(x)$.
\end{thm}

Blumberg and Hill do not give a detailed proof in \cite{BH:ITF}, but assert that the
proof strategy of \cite{HH:EqvarMultClosure} can be mimicked in the setting of incomplete Tambara functors. For completeness, we
include a (different and more elementary) proof here.

\begin{proof} 
As before, we decorate the structure maps of the localization with a tilde. In order to simplify notation, write $x_K := R^G_K(x)$
and similar for $H$. Fix an admissible set $H/K \in \cI(H)$. If $\tilde{N} := \tilde{N}_K^H$ exists, it must necessarily be given as
\numberwithin{equation}{section}
\setcounter{equation}{33}
\setcounter{definition}{34}
\begin{equation} \label{formula N tilde}
\tilde{N} \left( \frac{a}{x_K^n} \right) = \frac{N(a)}{N(x_K)^n}.
\end{equation}
This expression is well-defined if and only if $N(x_K) \in
\uR(H)$ becomes a unit after inverting $x_H$, i.e., if and only if it divides a power of $x_H$. \\
For the ``if'' direction, the argument above shows that $\Loc{\uR}{x}$ is a Green ring equipped with norms
$\tilde{N}_K^H$ for all admissible sets $H/K \in \cI(H)$ and all $H \leq G$.
From (\ref{formula N tilde}) we see that the reciprocity relations \cite[Prop.~4.10, Prop.~4.11]{BH:ITF}
satisfied by the norms of $\uR$ imply the reciprocity relations for the norms of $\Loc{\uR}{x}$. Thus by
\cite[Thm.~4.13]{BH:ITF}, $\Loc{\uR}{x}$ is a $\cI$-Tambara functor. Moreover, the canonical map $\uR \to
\Loc{\uR}{x}$ is a map of $\cI$-Tambara functors.
One readily verifies that the unique ring maps out of $\Loc{\uR(H)}{x_H}$ given by the universal properties for varying $H \leq
G$ assemble into a map of $\cI$-Tambara functors which exhibits $\Loc{\uR}{x}$ as the localization of $\uR$ at $x$. \\
The ``only if'' direction follows from the observation after (\ref{formula N tilde}) and the fact that
restriction is multiplicative.
\end{proof}

As before, this applies to localizations which invert natural numbers.

\begin{cor} \label{cor inverting primes alg}
Let $n \in \Z$, viewed as the element $n \cdot \left[ G/G \right] \in A(G)$. Then $\Abl \lbrack \frac{1}{n} \rbrack$ is a
complete Tambara functor.
\end{cor}

In particular, Question \ref{main Q alg} also makes sense for the local variants of the Burnside ring.
\section{Idempotent splittings of the Burnside ring and the $G$-sphere spectrum} \label{section:splitting}
We review Dress's classification of idempotents in the ($\fp$-local) Burnside ring and
describe the resulting product decompositions of the Burnside Mackey functor and the $G$-equivariant sphere spectrum. All of the statements in this section are easy consequences
of Dress's result and are probably well-known to the experts. The author does not claim any originality for these results.

\subsection{Idempotents in the Burnside ring} \label{subsection:dress}
Let $\fp$ be a collection of prime numbers and set $\Z_{(\fp)} := \Z \left[ p^{-1} \, | \, p \notin \fp \right]$. If $\fp$ is the
collection of all primes, nothing is inverted and hence $\Z_{(\fp)} = \Z$. If $\fp$ is the empty set, then all primes are inverted,
hence $\Z_{(\fp)} = \Q$.
For $\fp = \{ p \}$, we
obtain the usual $p$-localization $\Z_{(\fp)} = \Z_{(p)}$, which justifies the notation. Write $A(G)_{(\fp)} := A(G) \otimes_\Z
\Z_{(\fp)}$ for the $\fp$-local Burnside ring.

\begin{lemma}
Every finite group $G$ has a unique minimal normal subgroup $\Opi{G}$ such that the quotient $G/\Opi{G}$ is a solvable $\fp$-group,
i.e.~a solvable group whose order is only divisible by primes in $\fp$.
\end{lemma}

\begin{proof}
If $\fp$ is the collection of all primes, this is \cite[Prop.~1]{tomDieck:idempotent}; the same proof applies
for any choice of $\fp$.
\end{proof}

\begin{definition}
The group $\Opi{G} \leq G$ is called the $\fp$-\emph{residual subgroup} of $G$. A group $G$ is called $\fp$-\emph{perfect} if $G =
\Opi{G}$. If $\fp$ contains all primes, we will write $\Osolv{G} := \Opi{G}$ for the minimal normal subgroup with solvable quotient.
\end{definition}

\begin{rem}
The following statements are easily verified.
\begin{enumerate}[i)]
  \item For $\fp = \{ \mathrm{ all \, primes } \}$, the notion of a $\fp$-perfect group agrees with that of a
  perfect group (in the usual sense of group theory).
  \item For $\fp = \{ p \}$, the group $\Opi{G}$ is known to group theorists as the $p$-residual subgroup $O^p(G)$ and the condition
  that the quotient be solvable is redundant since every finite $p$-group is solvable.
  \item For $\fp = \emptyset$, every finite group $G$ is $\fp$-perfect because the trivial group is the only $\fp$-group.
\end{enumerate}
\end{rem} 

The following classification result is due to Dress. Recall that the assignment $S \mapsto |S^H|$ given by
taking the cardinality of the $H$-fixed points of a finite $G$-set $S$ extends to an injective ring homomorphism
\[ (\phi^H)_{(H) \leq G} \colon A(G) \to \prod_{(H) \leq G} \Z, \]
where the product is taken over conjugacy classes of subgroups $H \leq G$ \cite[(4), (5), Lemma 1]{dress:solvable}. The same is
true after inverting primes since $\Z_{(\fp)}$ has no torsion. The number $\phi^H(x)$ is called the \emph{mark} of $x$ at $H$.

\begin{thm}[\cite{dress:solvable}, Prop.~2] \label{Dress idempotents}
There is a bijection between the conjugacy classes of $\fp$-perfect subgroups $L \leq G$ and the set of primitive idempotent
elements of $\AGpi$ which sends $L$ to the element $e_L \in \AGpi$ whose mark $\phi^H(e_L)$ at a subgroup
$H \leq G$ is one if $\Opi{H} \sim L$ are conjugate in $G$, and zero otherwise.
\end{thm}

\begin{rem}
It follows immediately that $G$ is solvable if and only if $A(G)$ does not have any non-trivial idempotents. This originally
motivated Dress's work in \cite{dress:solvable}.
\end{rem}

\begin{rem}
Note that if $p$ does not divide the order $G$, then all subgroups $L \leq G$ are $p$-perfect, hence all idempotents of $A(G)
\otimes \Q$ are contained in the subring $A(G)_{(p)}$. For the other extreme case, if $G$ is a $p$-group, then only the trivial
subgroup is $p$-perfect, hence the only idempotents in $A(G)_{(p)}$ are zero and one.
\end{rem}

\subsection{Idempotent splittings of the Burnside ring} \label{subsection:splitting Burnside}
For any commutative ring, decomposing $1$ into a finite sum of idempotents yields a product
decomposition.

\begin{cor} \label{cor splitting A(G)}
The product of the canonical maps to the localizations
\[ \AGpi \to \prod_{(L) \leq G} \Loc{\AGpi}{e_L} \]
is an isomorphism of rings. Here, the product is taken over conjugacy classes of perfect subgroups of $L \leq
G$.
\end{cor}

One readily verifies that the statement above can be upgraded to a splitting of Green rings, where for any subgroup $H \leq G$, we
view $e_L$ as an element of $\AGpi$ via the restriction map $R_H^G \colon \AGpi \to A(H)_{(\fp)}$.

\begin{notation}
For brevity, we will write $\Loc{\Ablpi}{e_L}$ for the levelwise localization $\Loc{\Ablpi}{R_{(-)}^G(e_L)}$, see Section
\ref{subsection:preservation alg}.
\end{notation}

\begin{prop} \label{Dress splitting}
The product of the canonical maps to the localizations
\[ \Ablpi \to \prod_{(L) \leq G} \Loc{\Ablpi}{e_L} \]
is an isomorphism of Green rings.
\end{prop}

The left hand side is even a Tambara functor. Question \ref{main Q alg} asks whether the factors on the right
hand side inherit norms from $\Ablpi$, and whether the splitting preserves these norms.

\begin{rem} \label{rem non-zero localization}
The value of the Green ring $\Loc{\Ablpi}{e_L}$ at a subgroup $K \leq G$ is non-zero
if and only if $L$ is subconjugate to $K$, as follows from the description of $e_L$ in terms of
marks in Theorem \ref{Dress idempotents}.
\end{rem}

\begin{rem} \label{tilde notation}
Note that for any idempotent $e \in \AGpi$, the localization $\Loc{\AGpi}{e}$ is canonically isomorphic to the submodule $e \cdot
\AGpi$. The restriction maps $\tilde{R}_K^H$ and transfer maps $\tilde{T}_K^H$ of $\Loc{\Ablpi}{e}$ are given by the formulae \\
\[ \tilde{R}^H_K(R^G_H(e) \cdot a) := R^G_K(e) \cdot R^H_K(a) \]
and
\[ \tilde{T}_K^H(R^G_K(e) \cdot b) := R^G_H(e) \cdot T_K^H(b) \]
for all $a \in A(H)$ and $b \in A(K)$, where $R$ and $T$ denote the restrictions and transfers of $\Ablpi$
(cf.~\cite[Thm.~V.4.6]{LMS}). The equations that go into verifying Proposition~\ref{Dress splitting} can
easily be read off from these formulae.
The analogous $\fp$-local statements hold.
\end{rem}

\begin{rem}
We warn the reader that even though any restriction of $e_L$ to a proper subgroup $H \leq G$ is still an idempotent, it will in
general not be primitive. More precisely, it splits as an $n$-fold sum of primitive idempotents of $A(H)_{(\fp)}$ where $n$ is the
number of $H$-conjugacy classes contained in the $G$-conjugacy class of $L$. \nodetails{}{\redfont{LATER: Check this!}}
\end{rem}

\subsection{Idempotent splittings of the sphere spectrum} \label{subsection:splitting sphere}
We now turn to the homotopical consequences of the above splitting.
First recall the following theorem which goes back to Segal \cite[Cor. of Prop.~1]{segal:ESHT}.

\begin{thm}[See \cite{schwede:lecture_ESHT_2018}, Thm.~6.14, Ex.~10.11] \label{segals theorem}
For all $H \leq G$, there is a ring isomorphism $A(H) \to \pi_0^H(\sphere)$ which sends the
class represented by $H/K$ to the element $T_K^H(\id)$.
For varying $H$, these maps assemble into an isomorphism of Tambara functors
$\Abl \cong \underline{\pi}_0(\sphere)$.
\end{thm}

\begin{rem}
The isomorphism $\Abl \cong \underline{\pi}_0(\sphere)$ is completely determined by the
requirements that it be unital and respect transfers.
\end{rem}

Dress's classification of idempotent elements then immediately implies the next statement.

\begin{prop} \label{htpy Dress splitting}
The product of the canonical maps to the localizations is a weak equivalence of $\fp$-local
$G$-spectra
\[ \sphere_{(\fp)} \simeq \prod_{(L) \leq G} \Loc{\sphere_{(\fp)}}{e_L} \]
where the product is taken over conjugacy classes of $\fp$-perfect subgroups.
For any naive $N_\infty$-operad $\cO$, i.e., any $N_\infty$ operad whose homotopy type is
the unique minimal element in the poset of $N_\infty$ ring structures, this is a splitting of
$\mathcal{O}$-algebras (up to equivalence of $G$-spectra).
\end{prop}

The $\cO$-action on $\sphere_{(\fp)}$ in the last part of Proposition~\ref{htpy Dress splitting} is the
canonical action that factors through the commutative operad, using that $\sphere_{(\fp)}$ admits the
structure of a commutative monoid in orthogonal $G$-spectra.

\begin{proof}
The fact that the $e_L$ form a complete set of orthogonal idempotents implies that the map induces
isomorphisms on all equivariant homotopy groups. Moreover, the canonical maps
$\sphere_{(\fp)} \to \Loc{\sphere_{(\fp)}}{e_L}$ are all maps of $\cO$-algebras, as
follows from the fact that localization always preserves naive $N_\infty$ rings, see Corollary~\ref{cor naive preservation}.
\end{proof}

\begin{rem} \label{rem tom Dieck}
Note that the splitting of Prop.~\ref{htpy Dress splitting} induces a splitting of $\pi_\ast^G(\sphere)$ that
is different from the tom Dieck splitting \cite[Satz~2]{tomDieck:orbittypenII}: the idempotent splitting of
$\pi_0^G(\sphere) \cong A(G)$ is the decomposition into simple ideals, while tom Dieck's splitting
corresponds to the basis of $A(G)$ given by the classes of the $G$-sets $G/H$, where $H$ runs over all
conjugacy classes of subgroups of $G$. In particular, the tom Dieck splitting is not multiplicative.
\end{rem}

\begin{rem}
It follows from Remark~\ref{rem non-zero localization} that the underlying $H$-spectrum of
$\Loc{\sphere_{(\fp)}}{e_L}$ is non-trivial if and only if $L$ is subconjugate in $G$ to $H$. In particular,
the underlying non-equivariant spectrum of $\Loc{\sphere_{(\fp)}}{e_L}$ is trivial for $L \neq 1$ and is the
non-equivariant sphere spectrum for $L = 1$. 
\end{rem}

Question \ref{main Q htpy} asks about the maximal $N_\infty$ ring structures on the localizations
$\Loc{\sphere_{(\fp)}}{e_L}$, and about the maximal $N_\infty$ ring structure preserved by the
splitting. The answer is given in Corollary~\ref{local cor htpy} (Corollary~\ref{cor htpy}) and
Corollary~\ref{local cor htpy splitting} (Corollary~\ref{cor htpy splitting}).
\section{Norms in the idempotent splittings} \label{section:results}
We state and prove the results which answer Question~\ref{main Q alg} and Question~\ref{main Q htpy}, including the local variants
where any collection of primes $\fp$ is inverted.

\subsection{Theorem A and consequences} \label{subsection:group theory}
The main combinatorial result of this paper is the following version of Theorem~\ref{Thm A}, stated in
full $\fp$-local generality:

\begin{thm} \label{local Thm A}
Let $\fp$ be a collection of primes. Let $L \leq G$ be a $\fp$-perfect subgroup and let $e_L \in
\AGpi$ be the corresponding primitive idempotent under the bijection from Theorem~\ref{Dress
idempotents}. Fix subgroups $K \leq H \leq G$. Then
the norm map $N_K^H \colon A(K)_{(\fp)} \to A(H)_{(\fp)}$ descends to a well-defined map of
multiplicative monoids
\[ \tilde{N}_K^H \colon \Loc{A(K)_{(\fp)}}{e_L} \to \Loc{A(H)_{(\fp)}}{e_L} \]
if and only if the following holds:
\begin{enumerate}
  \item[($\star$)] Whenever $L' \leq H$ is conjugate in $G$ to $L$, then $L'$ is contained in $K$.
\end{enumerate} 
\end{thm}

The characterization of $e_L$ in terms of marks in Theorem \ref{Dress idempotents} implies that
$R^G_H(e_L) = 0$ whenever $L$ is not subconjugate in $G$ to $H$. From this, it is clear that
the norm $\tilde{N}_K^H$ exists for trivial reasons if $K$ is not super-conjugate in $G$ to $L$: it is just the zero
morphism between zero rings.
Similarly, there cannot be a norm map $\tilde{N}_K^H$ inherited from $N_K^H$ if $K$ is not
super-conjugate to $L$, but $H$ is. Indeed, it would have to be a map of multiplicative monoids from the zero ring to a
non-trivial ring, hence would satisfy $\tilde{N}_K^H(0)=1$. But $N_K^H(0)=\left[\mathrm{map_K(H,\emptyset)}\right] =0$ before
localizing, which is a contradiction.
The other cases are not obvious. We defer the proof of Theorem \ref{local Thm A} to Section \ref{subsection:proof} and
first state and prove the locally enhanced versions of Corollary~\ref{cor alg normal} and Corollary~\ref{cor alg trivial}.

\begin{cor} \label{local cor alg normal}
Assume that $L \leq G$ is $\fp$-perfect. Then $L$ is normal in $G$ if and only if the summand $\Loc{\Ablpi}{e_L}$ inherits from
$\Ablpi$ all norms of the form $\tilde{N}_K^H$ such that $K$ contains a subgroup conjugate in $G$ to $L$.
\end{cor}

\begin{proof}
If $L$ is normal, it is the only group in its $G$-conjugacy class, hence the condition ($\star$) of Theorem~\ref{local Thm A} is
satisfied for such $K \leq H$. Conversely, if the condition holds for the groups $K:=L$ and $H:=G$,
then any $G$-conjugate of $L$ is contained in $L$, hence $L$ is normal in $G$.
\end{proof}

\begin{cor} \label{local cor alg trivial}
The Green ring $\Loc{\Ablpi}{e_L}$ inherits norms $\tilde{N}_K^H$ for all $K \leq H$ if and only if $L=1$ is
the trivial group. In this case, the norm maps equip $\Loc{\Ablpi}{e_1}$ with the structure of a Tambara functor.
\end{cor}

\begin{proof}
If $L=1$, then all groups are supergroups of $L$ and all subgroup inclusions give rise to norm
maps by Corollary~\ref{local cor alg normal}. It then follows from \cite[Thm.~4.13]{BH:ITF} that $\Loc{\Ablpi}{e_1}$ is a Tambara
functor, cf.~the proof of Theorem~\ref{thm alg preservation}.
Conversely, if $L$ is non-trivial $\fp$-perfect, the inclusion $1 \to G$ does not give
rise to a well-defined norm on $\Loc{\Ablpi}{e_L}$.
\end{proof}

\begin{rem}
The ``only if'' part is implicit in work of Blumberg and Hill, at least integrally: All idempotents $e_L$ different from $e_1$ lie
in the augmentation ideal of $\AG$. If inverting such an element yielded a Tambara functor, then it would have to be the zero
Tambara functor, see \cite[Example~5.25]{BH:ITF}. But $\Loc{\Abl}{e_L}$ is always non-zero.
\end{rem}

\begin{rem}
It is also implicit in Nakaoka's work on ideals of Tambara functors \cite{nakaoka:ideals} that the idempotent summands of the
($\fp$-local) Burnside ring Mackey functor cannot all be Tambara functors, for if they were, then the idempotent splitting would be
a splitting of Tambara functors. But by \cite[Prop.~4.15]{nakaoka:ideals}, this implies that $A(1) \cong \Z$ splits non-trivially,
which is absurd. (Note that there is a minor error in statements (2)--(4) of loc.~cit.: the
requirement that the respective ideals and elements be non-zero is missing.)
\end{rem}

\begin{rem}
When working $p$-locally, the ring $\Loc{A(G)_{(p)}}{e_1}$ can be described in two different ways: It agrees with the
$p$-local Burnside ring with $p$-isotropy, i.e., the $p$-localization of the Grothendieck ring of finite $G$-sets all of
whose isotropy groups are $p$-groups. Moreover, it can be identified with the $p$-localization of the Burnside ring of the
$p$-fusion system of the group $G$. We refer the reader to \cite[Section~5]{grodal:burnside} for details.
\end{rem}

As an illustration of Theorem \ref{local Thm A}, we will discuss the idempotent splittings of $A(A_5)$ (integrally) and
$A(\Sigma_3)$ (locally at the primes $2$ and $3$) in detail in Section \ref{section:examples}. There, we also spell out what
happens in the rational splitting ($\fp=\emptyset$) for any finite group $G$.

\subsection{The proof of Theorem A} \label{subsection:proof}
The main idea of the proof is that we can check the hypotheses for preservation of norm maps from Theorem~\ref{thm alg
preservation} on marks. As norm maps in the Burnside ring are given by co-induction functors of equivariant sets, we need to
understand how they interact with taking fixed points. To that end, we first record some technical statements before giving the
proof of Theorem~\ref{local Thm A} (Theorem~\ref{Thm A}).

\begin{lemma}
For subgroups $K, H \leq N \leq G$, let $P$ be the pullback in the category of $G$-sets of the canonical surjections $G/H \to G/N$
and $G/K \to G/N$.
\begin{center}
$ \xymatrix@M=8pt@R=1.5pc{
	P \ar[d] \ar[r] & G/H \ar[d] \\
	G/K \ar[r] & G/N
}$
\end{center}
Then $P$ has an orbit decomposition given by\\
\[ P \cong \coprod_{n \in K\backslash N/H} G/(K \cap {}^n \! H) \]
where the summation is over representatives of double cosets. Under this identification, the map $P \to G/K$
is the sum of canonical surjections associated to the subgroup inclusions $K \cap {}^n \! H \leq K$, whereas
the map $P \to G/H$ is given on the $n$-th summand by conjugation by $n^{-1}$ followed by the canonical
surjection associated to the subgroup inclusion $({}^{n^{-1}} K) \cap H \leq H$. \qed
\end{lemma}

This implies a multiplicative double coset formula for norm maps of $\Ablpi$.

\begin{cor} \label{multiplicative DCS}
For $K, H, N$ and $G$ as before and all $x \in A(H)_{(\fp)}$, the following identity holds in $A(K)_{(\fp)}$:
\[ R^N_K N_H^N(x) = \prod_{n \in K\backslash N/H} N_{K \cap {}^n \! H}^K c_n R^H_{({}^{n^{-1}} \! K)
\cap H}(x) \] where we wrote $c_n$ for the map induced from conjugation by $n \in N$.
\end{cor}

\begin{lemma} \label{lemma marks of norms}
The norms of $\Ablpi$ satisfy $\phi^H(N_K^H(a))=\phi^K(a)$ for all $a \in A(K)_{(\fp)}$ and all nested subgroups $K \leq H \leq G$.
\end{lemma}

\begin{proof}
One readily verifies that for a finite $H$-set $X$, evaluation at the unit defines a bijection
\[ \Hom_H(G,X)^G \cong X^H. \]
It follows that the statement is true for all actual $H$-sets and hence for the submonoid of all
$\Z_{(\fp)}$-linear combinations of $H$-sets with non-negative coefficients. Since the latter
submonoid generates the group $A(H)_{(\fp)}$ and norm maps are algebraic (see
\cite[Prop.~4.7]{tambara}), \cite[Lemma~4.5]{tambara} implies that the statement is true for all virtual
$H$-sets.
\end{proof}

\begin{cor}[Cf.~\cite{oda}, Lemma~2.2] \label{marks of a norm formula}
For $Q,K \leq H$ and $x \in A(K)_{(\fp)}$, we have:
\[ \phi^Q N_K^H(x) = \prod_{h \in Q \backslash H/K} \phi^{{\,}^{h^{-1}} \! Q \cap K}(x) \]  
\end{cor}

\begin{proof}
In the following computation, the second equality is the multiplicative double coset formula of Corollary \ref{multiplicative DCS},
and the third uses that $\phi^Q$ is a ring homomorphism. The fourth equality is an application of Lemma~\ref{lemma marks of norms}.
\\
\[ \phi^Q N_K^H(x) = \phi^Q R^H_Q N_K^H(x)
= \phi^Q \left( \prod_{h} N_{Q \cap {\,}^h \! K}^Q c_h R^K_{{\,}^{h^{-1}} \! Q \cap K}(x) \right) \]
\\
\[ = \prod_{h} \phi^Q N_{Q \cap {\,}^h \! K}^Q c_h R^K_{{\,}^{h^{-1}} \! Q \cap K}(x)
= \prod_{h} \phi^{Q \cap {\,}^h \! K} c_h R^K_{{\,}^{h^-1} \! Q \cap K}(x) 
= \prod_{h} \phi^{{\,}^{h^{-1}} \! Q \cap K}(x) \qedhere \]
\end{proof}

\begin{lemma} \label{lemma little lemma}
Let $e, e' \in R$ be idempotents in a commutative ring. Then $e$ divides $e'$ if and only if $e \cdot e' =
e'$.
\end{lemma}

\begin{proof}
Assume that $e$ divides $e'$. Then $e' \in eR$, hence $e \cdot e' = e'$, since multiplication by $e$ is
projection onto the idempotent summand $eR$ of $R$. The other direction is obvious.
\end{proof}

\begin{lemma} \label{basic facts residual subgroup}
For $H \leq G$ and $g \in G$, the following holds:
\begin{enumerate}[a)]
  \item $\Opi{H} \subseteq \Opi{G}$
  \item $\Opi{^gH} = {}^g(\Opi{H})$
\end{enumerate}
\end{lemma}

The author learned the proof of part a) from Joshua Hunt.

\begin{proof}
Since $\Opi{G}$ is normal in $G$, we know that $H \cap \Opi{G}$ is normal in $H$. Now the group $H/(H \cap \Opi{G}) \cong
(H \cdot \Opi{G})/\Opi{G} \leq G/\Opi{G}$ is isomorphic to a subgroup of a solvable $\fp$-group, hence is a solvable $\fp$-group
itself. By minimality, $\Opi{H} \leq H \cap \Opi{G} \leq \Opi{G}$, which proves a). \\
The assertion b) follows from the fact that conjugation by $g$ induces a bijection between the subgroup lattices of $H$ and
$^gH$ which preserves normality.
\end{proof}

\begin{prop} \label{prop technical heart}
In the situation of Theorem~\ref{local Thm A}, the following are equivalent:
\begin{enumerate}
  \item[($\star$)] Every subgroup $L' \leq H$ that is conjugate in $G$ to $L$ is contained in $K$.
  \item[($\Diamond$)] For all $Q \leq H$ such that $\Opi{Q} \sim_G L$, we have $\phi^Q(N_K^H(R_K(e_L))) = 1$.
\end{enumerate}
\end{prop}

\begin{proof}
The proof proceeds in three steps. Step 1 and 2 simplify the condition ($\Diamond$), whereas Step 3 shows
that the resulting reformulation of ($\Diamond$) is equivalent to ($\star$). \\
\emph{Step 1:} Let $Q \leq Q' \leq H$ such that $\Opi{Q} \sim_G L \sim_G \Opi{Q'}$. We claim that if the
conclusion of ($\Diamond$) holds for $Q$, then it does so for $Q'$. \\
Indeed, if the conclusion of ($\Diamond$) holds for $Q$, then Theorem~\ref{Dress idempotents} together with
Corollary~\ref{marks of a norm formula} implies that for all $h \in H$, we have $\Opi{Q \cap {\,}^h \! K}
\sim_G L$. From Lemma~\ref{basic facts residual subgroup}, we see that
\[ L \sim_G \Opi{Q \cap {\,}^h \! K} \leq \Opi{Q' \cap {\,}^h \! K} \leq \Opi{Q'} \sim_G L, \]
so $\Opi{Q' \cap {\,}^h \! K}$ is conjugate to $L$, and hence the conclusion of ($\Diamond$) holds for $Q'$.
\\
The above claim shows that when verifying ($\Diamond$), we need not take into account all elements of the set $\{ Q \leq H \, | \,
\Opi{Q} \sim_G L \}$ but can restrict attention to its minimal elements under inclusion, i.e., to the groups $L'' \leq H$ such that
$L'' \sim_G L$.
In other words, ($\Diamond$) is equivalent to:
\begin{enumerate}
  \item[($\Diamond a$)] For all $L'' \leq H$ such that $L'' \sim_G L$, we have $\phi^{L''}(N_K^H(R_K(e_L))) = 1$.
\end{enumerate}
\emph{Step 2:} Let $L''$ be as in the assumption of ($\Diamond a$). As we have seen, the equation
$\phi^{L''}(N_K^H(R_K(e_L))) = 1$ holds if and only for all $h \in H$, we have $\Opi{L'' \cap {\,}^h \! K} \sim_G L$. But
\[ \Opi{L'' \cap {\,}^h \! K} = \Opi{{\,}^{h^{-1}} \! L'' \cap K}, \]
so substituting $L'$ for ${\,}^{h{^-1}} \! L''$ shows that ($\Diamond$) is equivalent to:
\begin{enumerate}
  \item[($\Diamond b$)] For all $L' \leq H$ such that $L' \sim_G L$, we have $\Opi{L' \cap K} \sim_G L$.
\end{enumerate}
\emph{Step 3:} We are left to show that for $L'$ as in the assumption of ($\Diamond b$), $L'$ is in $K$ if and
only $\Opi{L' \cap K} \sim_G L$. For the ``only if'' part, assume that $L' \leq K$, then
$ \Opi{L' \cap K} = L' \sim_G L.$
For the ``if'' part, observe that
\[ L \sim_G \Opi{L' \cap K} \leq L' \cap K \leq L'. \]
The conjugate copy of $L$ contained in $L' \cap K$ must be $L'$, so $L' \leq K$.
\end{proof}
 
\begin{proof}[Proof of Theorem~\ref{local Thm A}]
We know from Theorem~\ref{thm alg preservation} that the norm $N_K^H$ descends to a well-defined map $\tilde{N}_K^H$ if and only if
the element $N_K^H(R_K(e_L))$ divides $R_H(e_L)$ in $A(H)_{(\fp)}$. By Lemma~\ref{lemma little lemma}, this division relation is
equivalent to the equation
\[ N_K^H(R_K(e_L)) \cdot R_H(e_L) = R_H(e_L) \]
and holds if and only if for all $Q \leq H$, we have
\[ \phi^Q(N_K^H(R_K(e_L))) \cdot \phi^Q(R_H(e_L)) = \phi^Q(R_H(e_L)). \]
Here, we used that the homomorphism of marks
\[ (\phi^Q)_{(Q) \leq H} \colon A(H)_{(\fp)} \to \prod_{(Q) \leq H} \Z_{(\fp)} \]
is an injective ring homomorphism. \\
All three integers in the last equation are idempotents, hence can only be $0$ or $1$, and the equation holds in all cases except
when $\phi^Q(N_K^H(R_K(e_L)))$ is zero, but $\phi^Q(R_H(e_L))$ is one. The formula for marks given in Theorem~\ref{Dress idempotents}
then implies that the equation is equivalent to the condition ($\Diamond$) of Proposition~\ref{prop technical heart}. The latter is
equivalent to ($\star$), and Theorem~\ref{local Thm A} follows.
\end{proof}

\subsection{The incomplete Tambara functor structure} \label{subsection:structure}
It still remains to see how the collection of norm maps described by Theorem \ref{local Thm A} fits into the framework of
\cite{BH:ITF}. First of all, we describe the norm maps in $\Loc{\Ablpi}{e_L}$ arising from arbitrary maps of $G$-sets. This is the
special case $\uR = \Ablpi, \; x = e_L$ of the following result:

\begin{prop} \label{prop projections}
Let $\uR$ be a Tambara functor and let $x \in \uR(G)$ be some element. Let $f \colon X \to Y$ be any map of
finite $G$-sets.
Then the levelwise localization $\Loc{\uR}{x}$ inherits a norm map $\tilde{N}_f$ from $\uR$ if and only if for
all $x \in X$, it inherits a norm map $\tilde{N}_{{f}|{G \cdot x}}$ for the restriction
\[ \restr{f}{G \cdot x} \colon G \cdot x \to G \cdot f(x) \]
to the orbits of $x$ and $f(x)$.
\end{prop}

\begin{proof}
Choose an orbit decomposition $Y \cong \coprod_i G/H_i$ and set $X_i := f^{-1}(G/H_i)$.
Now $f$ can be written as a sum of maps
\[ f_i \colon X_i \to G/H_i \]
such that $f_i$ is the unique map $\emptyset \to G/H_i$ if $G/H_i \not\subseteq \im(f)$ and such that $f_i$ is
a sum of canonical surjections
\[ f_{ij} \colon G/K_{ij} \to G/H_i \]
induced by subgroup inclusions $K_{ij} \leq H_i$ if $G/H_i \subseteq \im(f)$. Under these identifications, the
maps $f_{ij}$ are precisely the restrictions of $f$ to orbits of $X$.\\
The proof now proceeds in three steps. \\
\emph{Step 1:}
By the universal property of the product (of underlying multiplicative monoids), a potential norm map defined
by $f$ is given componentwise by the potential norms induced by the maps $f_i \colon X_i \to G/H_i$.
Consequently, $\tilde{N}_f$ exists if and only if $\tilde{N}_{f_i}$ exists for all $i$.
\\
\emph{Step 2:} We claim that for all $i$ such that $G/H_i \not\subseteq \im(f)$, the norm $\tilde{N}_{f_i}$
associated to $f_i \colon \emptyset \to G/H_i$ exists unconditionally. Indeed, the map $N_{f_i} \colon
0 = \uR(\emptyset) \to \uR(G/H_i)$ is just the inclusion of the multiplicative unit of the ring $\uR(H_i)$, so
$\tilde{N}_{f_i}$ exists and is given as the inclusion of the multiplicative unit of the ring
$\Loc{\uR(H_i)}{(\Res^G_{H_i}(x))}$. Thus, the existence of $\tilde{N}_{f}$ only depends on the maps $f_i$ for
those $i$ such that $G/H_i \subseteq \im(f)$. \\
\emph{Step 3:} We are left to show that a map $f_i \colon \coprod_{j} G/K_{ij} \to G/H_i$ gives rise to a
norm map if and only if all of the maps $f_{ij} \colon G/K_{ij} \to G/H_i$ do. The coproduct $\coprod_{j}
G/K_{ij}$ is the product of the $G/K_{ij}$ in the category of bispans with $k$-th projection map given by
\begin{center}
$ \xymatrix@M=8pt@R=1.5pc{
	\coprod_{j} G/K_{ij} & G/K_{ik} \ar@{_(->}[l] \ar[r]^{\id} & G/K_{ik} \ar[r]^{\id} & G/K_{ik}
}$
\end{center}
(see \cite[Prop.~7.5 (i)]{tambara}), so under the identification
$\uR(\coprod_{j} G/K_{ij}) \cong \prod_{j} \uR(K_{ij})$, the norm $N_{f_i}$ is of the form
\[ \prod_{j} \uR(K_{ij}) \to \uR(H_i), \quad (a_j)_j \mapsto \prod_j N_{f_{ij}}(a_j). \]
The analogous statement holds for the norms of $\Loc{\uR}{x}$, provided they exist. Thus, $\tilde{N}_{f_i}$ exists if and only
if $\tilde{N}_{f_{ij}}$ exists for all $j$.
\end{proof}

We would like to use Theorem~\ref{thm alg preservation} in order to show that $\Loc{\Ablpi}{e_L}$ is an incomplete Tambara functor
with norms as described in Theorem~\ref{local Thm A}. However, Theorem~\ref{thm alg preservation} is a statement about Tambara
functors structured by indexing systems, or equivalently (see Theorem~\ref{iso IS D}), structured by wide, pullback-stable, finite
coproduct-complete subcategories $D \subseteq \Set^G$. Thus, we first need to see that the maps $f$ which give rise to norm maps
form such a category $D$.

\begin{definition} \label{def D_L}
For a $\fp$-perfect subgroup $L \leq G$, let $D_L \subseteq \Set^G$ be the wide subgraph consisting of all
the maps of finite $G$-sets $f \colon X \to Y$ such that the orbit $G_{f(x)}/G_x$ obtained from stabilizer
subgroups satisfies the conditions of Theorem~\ref{local Thm A} for all points $x \in X$.
\end{definition}

\begin{prop} \label{D_L comes from I_L}
The subgraph $D_L$ is a wide, pullback-stable, finite coproduct-complete subcategory of $\Set^G$, hence corresponds to an indexing
system $\cI_L$ under the equivalence of posets of Theorem~\ref{iso IS D}.
\end{prop}

Explicitly, the admissible $H$-sets in $\cI_L$ are the objects over $G/H$ in $D_L$, see \cite[Lemma~3.19]{BH:ITF}. The three lemmas
below constitute the proof.

\begin{lemma}
The graph $D_L$ is a wide subcategory of $\Set^G$.
\end{lemma}

\begin{proof}
It is wide by definition and clearly contains all identities. Once we have shown that it is closed
under composition, associativity follows from associativity in $\Set^G$. Let $f \colon S \to T$
and $g \colon T \to U$ be admissible maps of $G$-sets. By Proposition~\ref{prop projections},
we may assume that $S=G/A, T=G/B$ and $U=G/C$ are transitive $G$-sets for nested subgroups $A \leq
B \leq C \leq G$, and $f,g$ are the canonical surjections. Thus, it suffices to show that if $C/B$ and $B/A$ are admissible, so is
$C/A$. This is immediate from the condition ($\star$) given in Theorem~\ref{local Thm A}.
\end{proof}

\begin{lemma}
The subcategory $D_L$ is finite coproduct-complete.
\end{lemma}

\begin{proof}
This follows directly from Proposition~\ref{prop projections}.
\end{proof}

\begin{lemma}
The subcategory $D_L$ is pullback-stable.
\end{lemma}

\begin{proof}
The problem reduces to canonical surjections between orbits by Proposition~\ref{prop projections}.
We have to show that if the canonical surjection $G/K \to G/H$ in the following pullback diagram is admissible, then so is its
pullback along the canonical map $G/A \to G/H$, where $A, K \leq H$ are subgroups.
\begin{center}
$ \xymatrix@M=8pt@R=1.5pc{
	P \ar[d] \ar[r] & G/K \ar[d] \\
	G/A \ar[r] & G/H
}$
\end{center}
This in turn amounts to verifying the condition ($\star$) of Theorem~\ref{local Thm A} for all summands of
\[ R^H_A(H/K) \cong \coprod_{\lbrack h \rbrack \in A\backslash H/K} A/(A \cap {\,}^h \! K). \]
Note that since $H/K$ is admissible, so are the isomorphic $H$-sets $H/{\,}^h \! K$ for all $h \in H$. Fix $L' \leq A$ such that
$L' \sim_G L$. We have to show that $L' \leq A \cap {\,}^h \! K$. But $L'$ is in $H$ and $H/{\,}^h \! K$ is admissible, so $L' \leq
{\,}^h \! K$ and hence $L' \leq A \cap {\,}^h \! K$.
\end{proof}

We obtain (the locally enhanced) Theorem \ref{Thm D}:

\begin{thm} \label{local Thm D}
Let $\fp$ be a collection of primes. Let $L \leq G$ be a $\fp$-perfect subgroup and let $e_L \in \AGpi$ be the
corresponding primitive idempotent.
Then the following hold:
\begin{enumerate}[i)]
  \item The admissible sets for $e_L$ assemble into an indexing system $\cI_L$ such that $\Loc{\Ablpi}{e_L}$ is an $\cI_L$-Tambara
  functor under $\Ablpi$.
  \item In the poset of indexing systems, $\cI_L$ is maximal among the elements that satisfy i).
  \item The map $\Ablpi \to \Loc{\Ablpi}{e_L}$ is the localization of $\Ablpi$ at $e_L$ in the category of $\cI_L$-Tambara functors.
\end{enumerate}
\end{thm}

\begin{proof}
Proposition \ref{D_L comes from I_L} shows that $\cI_L$ is an indexing system. Then the Green ring
$\Loc{\Ablpi}{e_L}$ equipped with the norm maps given by Theorem~\ref{local Thm A} is an $\cI_L$-Tambara
functor by \cite[Thm.~4.13]{BH:ITF}, see the proof of Theorem~\ref{thm alg preservation} for details. This
proves part~i). Theorem~\ref{thm alg preservation} also implies part iii).
Part ii) follows from Theorem~\ref{local Thm A} together with Proposition~\ref{prop projections}.
\end{proof}

Finally, we describe the maximal incomplete Tambara functor structure which is preserved by the idempotent splitting of the Green
ring $\Ablpi$ stated in Proposition~\ref{Dress splitting}.

\begin{lemma}
The (levelwise) intersection of a finite number of indexing systems is an indexing system. \qed
\end{lemma}

\begin{notation}
Write $\cI$ for the indexing system
\[ \cI := \bigcap_{(L) \leq G} \cI_L \]
where the intersection is over all conjugacy classes of $\fp$-perfect subgroups of $G$, and the indexing systems $\cI_L$ are the
ones given by Theorem~\ref{local Thm D}.
\end{notation}

For each $\fp$-perfect $L \leq G$, the $\cI_L$-Tambara functor $\Loc{\Ablpi}{e_L}$ is an $\cI$-Tambara functor by forgetting 
structure. Theorem~\ref{local Thm A} provides an explicit description of the admissible sets of $\cI$.

\begin{lemma} \label{admissible sets of I}
Let $K \leq H \leq G$, then $H/K$ is an admissible set for $\cI$ if and only if for all $\fp$-perfect $L \leq
H$, $L$ is contained in $K$.
\qed
\end{lemma}

We can now restate Corollary~\ref{cor alg splitting}.

\begin{cor} \label{local cor alg splitting}
The localization maps $\Ablpi \to \Loc{\Ablpi}{e_L}$ assemble into an isomorphism of $\cI$-Tambara functors
\[ \Ablpi \to \prod_{(L) \leq G \; \fp \mathrm{-perfect}} \Loc{\Ablpi}{e_L}. \]
\end{cor}

\begin{proof}
It is an isomorphism of Green rings by Proposition~\ref{Dress splitting}. Moreover, each of the localization maps $\Ablpi \to
\Loc{\Ablpi}{e_L}$ is a map of $\cI$-Tambara functors, and the product in the category of $\cI$-Tambara functors is computed
levelwise, see \cite[Prop.~10.1]{strickland:tambara}.
\end{proof}

\begin{rem}
We point out a possible alternative to our proof of Corollary~\ref{local cor alg splitting}. Blumberg and Hill generalized parts of
Nakaoka's theory of ideals of Tambara functors \cite{nakaoka:ideals} to the setting of incomplete Tambara functors, see
\cite[Section~5.2]{BH:ITF}. The author is confident that one could similarly generalize Nakaoka's splitting result
\cite[Prop.~4.15]{nakaoka:ideals} to the incomplete setting. It would state that an $\cI$-Tambara functor
$\uR$ splits non-trivially as a product of $\cI$-Tambara functors if and only if for each admissible set $X$ of $\cI$, there are
non-zero elements $a, b \in \uR(X)$ such that $a + b = 1$ and $\langle a \rangle \cdot \langle b \rangle = 0$. Such a result would
reprove our Corollary~\ref{local cor alg splitting}, using that the restrictions of the primitive idempotents $e_L$ along admissible
maps never become zero. We leave the details to the interested reader.
\end{rem}

\subsection{The $N_\infty$ ring structure} \label{subsection:htpy}
We return to the situation of Question~\ref{main Q htpy}, lift our algebraic results to the category of $G$-spectra and prove the
locally enhanced versions of Corollary~\ref{cor htpy}, Corollary~\ref{cor htpy trivial} and Corollary~\ref{cor htpy splitting}.

Observe that for any $N_\infty$ operad $\cP$, the object $\sphere_{(\fp)}$ admits the structure of a commutative monoid in
orthogonal $G$-spectra, hence admits a natural $\cP$-algebra action that factors through the commutative operad.

\begin{cor} \label{local cor htpy}
Let $L \leq G$ be a $\fp$-perfect subgroup and let $e_L \in \upi_0^G(\sphere)$ be the associated idempotent. For any
$\Sigma$-cofibrant $N_\infty$ operad $\cO_L$ whose associated indexing system is $\cI_L$, the following hold:
\begin{enumerate}[i)]
  \item The $G$-spectrum $\Loc{\sphere_{(\fp)}}{e_L}$ is an $\cO_L$-algebra under $\sphere_{(\fp)}$.
  \item In the poset of homotopy types of $N_\infty$ operads, $\cO_L$ is maximal among the elements that satisfy i).
  \item The map $\sphere_{(\fp)} \to \Loc{\sphere_{(\fp)}}{e_L}$ is a localization at $e_L$ in the category of $\cO_L$-algebras.
\end{enumerate}
\end{cor}

The cofibrancy assumption does not impose an obstruction to the existence of
$\cO_L$, see Remark \ref{rem cofibrancy one}.

\begin{proof}
Part iii) of Theorem~\ref{local Thm D} combined with Theorem~\ref{thm alg preservation} show that certain
divisibility relations hold in the $\IL$-Tambara functor $\upi_0(\sphere) \cong \Ablpi$. Then
Proposition~\ref{prop translation White HH} guarantees that $e_L$-localization preserves $\OL$-algebras, which
implies statements i) and iii).
For part ii), assume that there is an $N_\infty$ operad $\cO'$
whose homotopy type is strictly greater than that of $\cO_L$ such that $\Loc{\sphere_{(\fp)}}{e_L}$ is an $\cO'$-algebra.
Then, by Theorem \ref{Brun relative}, its $0$-th equivariant homotopy forms a $\cI'$-Tambara functor for the
indexing system $\cI'$ corresponding to $\cO'$. But this contradicts the maximality proved in Corollary \ref{local Thm D}.
\end{proof}

The following local enhancement of Corollary \ref{cor htpy trivial} is a homotopical reformulation of Corollary \ref{local cor alg
trivial} (Corollary \ref{cor alg trivial}).

\begin{cor} \label{local cor htpy trivial}
The $G$-spectrum $\Loc{\sphere_{(\fp)}}{e_L}$ is a $G$-$E_\infty$ ring spectrum if and only if $L=1$ is the trivial group.
\end{cor}

In particular, we see that the idempotent splitting of $\sphere$ is far from being a splitting of $G$-$E_\infty$ ring spectra.
Locally at the prime $p$, Corollary~\ref{local cor htpy trivial} recovers a (yet unpublished) result of Grodal.

\begin{thm}[\cite{grodal:burnside}, Cor.~5.5] \label{thm grodal}
The $G$-spectrum $\Loc{\sphere_{(p)}}{e_1}$ is a $G$-$E_\infty$ ring spectrum.
\end{thm}

Finally, we state the homotopy-theoretic analogue of Corollary~\ref{local cor alg splitting} in order to describe the maximal
$N_\infty$-ring structure preserved by the $\fp$-local idempotent splitting of the sphere. It is the local reformulation of
Corollary~\ref{cor htpy splitting}.

\begin{cor} \label{local cor htpy splitting}
Let $\cO$ be a $\Sigma$-cofibrant $N_\infty$ operad realizing the indexing system $\cI = \cap_{(L)} \cI_L$.
Then the idempotent splitting
\[ \sphere_{(\fp)} \simeq \prod_{(L) \leq G} \Loc{\sphere_{(\fp)}}{e_L} \]
is an equivalence of $\cO$-algebras, where the product is taken over conjugacy classes of $\fp$-perfect subgroups.
\end{cor}

\begin{proof}
The splitting is an equivalence of $G$-spectra by Proposition~\ref{htpy Dress splitting}. Moreover,
all of the maps to the localizations are maps of $\cO$-algebras, as can be seen from \ref{local cor htpy}.
\end{proof}

Together, Corollary~\ref{local cor htpy} and Corollary~\ref{local cor htpy splitting} answer
Question~\ref{main Q htpy} completely, for any family of primes inverted.
\section{Examples} \label{section:examples}
We illustrate our results in the rational case, in the case of the alternating group $A_5$,
working integrally, and that of the symmetric group $\Sigma_3$, working $3$-locally.

\subsection{The rational case}
In the case when $\fp = \emptyset$ and hence $\Z_{(\fp)} = \Q$, the rational Burnside ring
$A(G)_\Q$ has exactly one primitive idempotent $e_L$ for each conjugacy class of subgroups $L \leq G$. The
incomplete Tambara functor structures of the idempotent summands $\Loc{A(-)_\Q}{e_L}$ depend on
the subgroup structure of $G$ as described by Theorem~\ref{local Thm A}. However, it is
immediately clear from Lemma~\ref{admissible sets of I} that the idempotent splitting is only a
splitting of Green rings, but not a splitting of $\cI'$-Tambara functors for any indexing system
$\cI'$ greater than the minimal one. This phenomenon is also discussed in
\cite[Section~7]{BGK:AlgebraicModelNaive}, and it is precisely the reason why their approach only provides
an algebraic model for the rational homotopy theory of naive $N_\infty$ ring spectra, but cannot
possibly account for any non-trivial Hill-Hopkins-Ravenel norms.

\subsection{The alternating group $A_5$}
It is well-known that $A_5$ is the smallest non-trivial perfect group. Thus, it is the smallest example of a group whose Burnside
ring admits a non-trivial idempotent splitting when working integrally. Indeed, the only perfect subgroups are $1$ and
$A_5$, and these give rise to idempotent elements $e_1, e_{A_5} \in A(A_5)$. Theorem \ref{Dress idempotents} implies that their
marks are given by
\[ \phi^H(e_{A_5}) = \begin{cases}
1, & H = A_5 \\
0, & H \neq A_5 \\
\end{cases} \]
and vice versa for $e_1$.
We know from Corollary \ref{local cor alg trivial} that $\Loc{A(A_5)}{e_1}$ is a complete Tambara functor. On the other hand,
$\Loc{A(H)}{e_{A_5}}$ is trivial unless $H = A_5$, hence there cannot be any norm maps $N_H^{A_5}$ for proper subgroups $H \leq
A_5$. Moreover, by Corollary~\ref{local cor alg splitting}, the idempotent splitting of $\Abl$ is a splitting of $\cI_{A_5}$-Tambara
functors, i.e., it only preserves norms between proper subgroups. \\
By Corollary~\ref{local cor htpy} and Corollary~\ref{local cor htpy splitting}, the analogous statements hold for the $N_\infty$
ring structures on $\Loc{\sphere}{e_1}$ and $\Loc{\sphere}{e_{A_5}}$. Just like Example \ref{example hill}, this provides another
instance of the phenomenon that inverting a single homotopy element does not preserve any of the Hill-Hopkins-Ravenel norm maps
from proper subgroups to the ambient group. \\
Of course, all of this holds for any perfect group $G$ whose only perfect subgroup is the trivial
group.

\subsection{The symmetric group $\Sigma_3$ at the prime $3$}
Since $\Sigma_3$ is solvable, its Burnside ring $A(\Sigma_3)$ does not have any idempotents other than zero and one. We can obtain
interesting idempotent splittings by working locally at primes $p$ dividing the group order. All $2$-perfect subgroups of $\Sigma_3$
are normal, hence the case $p=2$ is completely covered by Corollary~\ref{local cor alg normal} and we only discuss the more
interesting case $p=3$ in detail.

Any map in the orbit category can be factored as an isomorphism followed by a canonical surjection, hence the admissibilty of
$\Sigma_3/H$ just depends on the conjugacy class of $H$ and we can just write $C_2$ for any of the three conjugate subgroups of
order two. Note that the $3$-residual subgroups $O^3(H)$ for $H \leq \Sigma_3$ are given as follows:
\[ \O^3(H) = \begin{cases}
\Sigma_3, & H=\Sigma_3 \\
1, & H=A_3 \\
C_2, & H=C_2 \\
1, & H=1
\end{cases} \]
Thus, all subgroups of $\Sigma_3$ except for $A_3$ are $3$-perfect. All subgroups of order two are conjugate in $\Sigma_3$, so there
are three idempotent elements in $A(\Sigma_3)_{(3)}$, corresponding to the conjugacy classes of
the $3$-perfect subgroups $1, C_2$ and $\Sigma_3$. In terms of marks, they are given as
\begin{center}
\begin{tabular}{c|c|c|c}
Subgroup $H \leq \Sigma_3$ & $\phi^H(e_1)$ & $\phi^H(e_{C_2})$ & $\phi^H(e_{\Sigma_3})$ \\ \hline
$1$        & 1 & 0 & 0 \\
$C_2$      & 0 & 1 & 0 \\
$A_3$      & 1 & 0 & 0 \\
$\Sigma_3$ & 0 & 0 & 1
\end{tabular}
\end{center}
The localization $\Loc{A(-)_{(3)}}{e_1}$ admits all norms by Corollary \ref{local cor alg trivial}. The norm maps of
$\Loc{A(-)_{(3)}}{e_{\Sigma_3}}$ are described by Corollary \ref{local cor alg normal}. In detail, this Mackey functor is zero at
all proper subgroups, but non-trivial at $\Sigma_3$. Consequently, there are norm maps $\tilde{N}_K^H$ if and only if $H$ and hence
$K$ is a proper subgroup of $\Sigma_3$, but all of these norms are maps between trivial rings.
\\
It remains to describe the idempotent localization $\Loc{A(-)_{(3)}}{e_{C_2}}$. The left hand diagram below
depicts the subgroups $H \leq \Sigma_3$ (up to conjugacy) and their inclusions. The right hand diagram
displays the ranks (as free $\Z_{(3)}$-modules) of the corresponding values of $\Loc{A(-)_{(3)}}{e_{C_2}}$ at
the subgroup $H$.
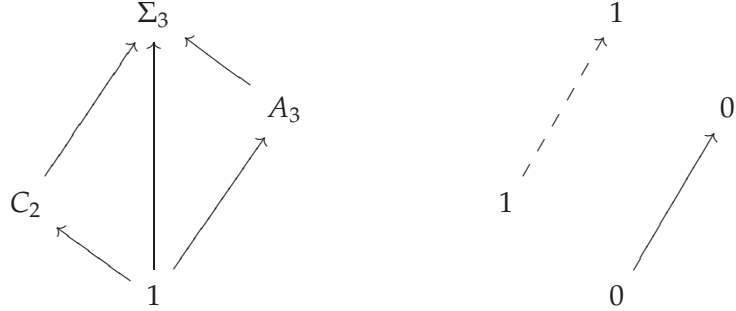
\begin{center}
$ \xymatrix@M=6pt@R=1.2pc{
	\,           & \Sigma_3         & \,          &\,& \,               & 1          & \,        \\
	\,           & \,               & A_3 \ar[ul] &\,& \,               & \,         & 0         \\
	C_2 \ar[uur] & \,               & \,          &\,& 1 \ar@{-->}[uur] & \,         & \,        \\
	\,           & 1 \ar[ul]
	              \ar[uuu] \ar[uur] & \,          &\,& \,               & 0 \ar[uur] & \,        \\
}$
\end{center}
There is a norm map from $1$ to $A_3$ for trivial reasons (indicated by the solid arrow) and the only other norm maps which
could potentially exist would be the maps $\tilde{N}_{C_2}^{\Sigma_3}$ where $C_2$ is any subgroup of order two (indicated by the
dashed arrow). However, if we choose $K=L=(12)$ and let $L' = (13)$, then the condition ($\star$)
of Theorem \ref{local Thm A} is not satisfied. Indeed, $L'$ is conjugate to $L$, but not contained in $K$. Consequently,
there is no norm map $\tilde{N}_{C_2}^{\Sigma_3}$. \\
We see from Lemma~\ref{admissible sets of I} that $\cI = \cI_{C_2}$, so in this case the idempotent splittings of $A(-)_{(3)}$ and
hence $\sphere_{(3)}$ only preserve the norm map $N_1^{A_3}$.
\section{Applications} \label{section:applications}

\subsection{Norm functors in the idempotent splitting of $\SpG$}
Any $G$-spectrum $X$ is a module over the sphere spectrum, hence admits an idempotent splitting
\[ X \simeq \prod_{(L) \leq G} \Loc{X}{e_L} \]
where $\Loc{X}{e_L}$ is the sequential homotopy colimit along countably many copies of the map
$X \cong X \wedge \sphere \stackrel{\id \wedge e_L}{\longrightarrow} X \wedge \sphere \cong X$. Thus, the
idempotent elements of $\AG$ induce a product decomposition of the category of $G$-spectra $\SpG$ by breaking
it up into categories of modules over the idempotent summands $\Loc{\sphere}{e_L}$. This splitting is
homotopically meaningful as it can be upgraded to a Quillen equivalence of model categories; similar
statements hold in the local cases. See~\cite[Thm.~4.4, Section~6]{barnes:splitting} for a more thorough
discussion. While this only depends on the additive splitting of Proposition~\ref{htpy Dress splitting}, some
additional multiplicative structure is present.

It is useful to consider not just the category of ($\fp$-local) $G$-spectra, but rather the symmetric monoidal categories of
($\fp$-local) $H$-spectra for all subgroups $H \leq G$ together with their restriction and norm functors. This kind of structure has
been studied in \cite{HH:EqvarSymMon, BH:modules-v2} under the name of $G$-\emph{symmetric monoidal
categories}.
From this perspective, Theorem~\ref{Thm A} measures the failure of the idempotent splitting of $\SpG$ to give rise to a splitting of $G$-symmetric
monoidal categories. Indeed, the factors only admit some of the Hill-Hopkins-Ravenel norm functors and hence form ``incomplete
$G$-symmetric monoidal categories'':
\nodetails{}{\redfont{LATER: Explain a bit more?}}

\begin{cor} \label{local norms modules}
Let $L \leq G$ be $\fp$-perfect and let $\cO_L$ be as in Corollary~\ref{local cor htpy}. Assume furthermore
that $\cO_L$ has the homotopy type of the linear isometries operad on a (possibly incomplete) universe $U$.
For all admissible sets $H/K$ of $\cI_L$, there are norm functors
\[ {~}_{\Res_H(\Loc{\sphere_{(\fp)}}{e_L})} N_{K, \Res_K(U)}^{H, \Res_H(U)} \colon
Mod(\Res^G_K(\Loc{\sphere_{(\fp)}}{e_L})) \to Mod(\Res^G_H(\Loc{\sphere_{(\fp)}}{e_L})) \]
built from the smash product relative to $\Loc{\sphere_{(\fp)}}{e_L}$
which satisfy a number of relations analogous to those for the norm functor $\mathrm{Sp}^H \to \SpG$, stated
in \cite[Thm.~1.3]{BH:modules-v2}.
\end{cor}

This is an immediate application of \cite[Thm.~1.1, Thm.~1.3]{BH:modules-v2} to Corollary~\ref{local cor htpy}. We refer to
\cite{BH:modules-v2} for a detailed discussion of modules over $N_\infty$ ring spectra.

The reason for the ``linear isometries'' hypothesis is explained in the introduction to \cite{BH:modules-v2}. It is
expected that it is not necessary, and that the $\infty$-categorical tools developed in \cite{barwick:PHCT-intro} and its
sequels will remove this technical assumption.

\nodetails{}{\redfont{Give a characterization of those $\cI_L$ which correspond to linear isometries operads?}}

\subsection{Idempotent splittings of equivariant topological K-theory}
Our main questions, Question~\ref{main Q htpy} and Question~\ref{main Q alg}, can be asked for any $G$-$E_\infty$ ring spectrum and
its idempotent splitting, assuming there are only finitely many primitive idempotents and that these admit a suitably explicit
description. In the sequel \cite{boehme:idempotent-characters_v2}, we will answer the analogues of our main
questions for the $G$-equivariant complex topological K-theory spectrum $KU_G$ and its real analogue $KO_G$. It turns out that
the solution can be reduced to the one given here, but in order to see this, a careful analysis of the complex
representation ring and its relationship with the Burnside ring is required.

\phantomsection
\addcontentsline{toc}{section}{Bibliography}
\bibliographystyle{alpha}
{\footnotesize \bibliography{topology} }

\newcommand{\etalchar}[1]{$^{#1}$}
\begin{thebibliography}{EKMM97}

\bibitem[AB18]{angeltveit-bohmann}
V.~Angeltveit and A.~M. Bohmann.
\newblock {Graded Tambara functors}.
\newblock {\em J. Pure Appl. Algebra}, 222(12):4126--4150, 2018.

\bibitem[Bar09]{barnes:splitting}
D.~Barnes.
\newblock {Splitting monoidal stable model categories}.
\newblock {\em J. Pure Appl. Algebra}, 213(5):846--856, 2009.

\bibitem[BDG{\etalchar{+}}]{barwick:PHCT-intro}
C.~Barwick, E.~Dotto, S.~Glasman, D.~Nardin, and J.~Shah.
\newblock {Parametrized higher category theory and higher algebra: A general
  introduction}.
\newblock \texttt{arXiv:1608.03654}.

\bibitem[BGK19]{BGK:AlgebraicModelNaive}
D.~Barnes, J.~P.~C. Greenlees, and M.~K\c{e}dziorek.
\newblock An algebraic model for rational na\"{i}ve-commutative
  {$G$}-equivariant ring spectra for finite {$G$}.
\newblock {\em Homology Homotopy Appl.}, 21(1):73--93, 2019.

\bibitem[BHa]{bachmann-hoyois_v4}
T.~Bachmann and M.~Hoyois.
\newblock {Norms in motivic homotopy theory}.
\newblock \texttt{arXiv:1711.03061v4}.

\bibitem[BHb]{BH:modules-v2}
A.~J. Blumberg and M.~A. Hill.
\newblock {G-symmetric monoidal categories of modules over equivariant
  commutative ring spectra}.
\newblock \texttt{arXiv:1511.07363v2}.

\bibitem[BH15]{BH:OperMult}
A.~J. Blumberg and M.~A. Hill.
\newblock {Operadic multiplications in equivariant spectra, norms, and
  transfers}.
\newblock {\em Adv. Math.}, 285:658–708, 2015.

\bibitem[BH18]{BH:ITF}
A.~J. Blumberg and M.~A. Hill.
\newblock Incomplete {T}ambara functors.
\newblock {\em Algebr. Geom. Topol.}, 18(2):723--766, 2018.

\bibitem[B{\"o}h]{boehme:idempotent-characters_v2}
B.~B{\"o}hme.
\newblock {Idempotent characters and equivariantly multiplicative splittings of
  K-theory}.
\newblock \texttt{arXiv:1808.09832v2}.

\bibitem[B{\"o}h18]{boehme:thesis}
B.~B{\"o}hme.
\newblock {\em {Equivariant multiplications and idempotent splittings of
  G-spectra}}.
\newblock PhD thesis, {University of Copenhagen}, {2018}.

\bibitem[BP17]{bonventre-pereira}
P.~Bonventre and L.~A. Pereira.
\newblock {Genuine equivariant operads}.
\newblock 2017.
\newblock \texttt{arXiv:1707.02226v1}.

\bibitem[Bru07]{brun:eqvar-spectra}
M.~Brun.
\newblock {Witt vectors and equivariant ring spectra applied to cobordism}.
\newblock {\em {Proc. Lond. Math. Soc. (3)}}, 94:351–385, 2007.

\bibitem[Dre69]{dress:solvable}
A.~Dress.
\newblock { A characterisation of solvable groups}.
\newblock {\em Math. Z.}, 110:213--217, 1969.

\bibitem[EKMM97]{EKMM}
A.~D. Elmendorf, I.~Kriz, M.~A. Mandell, and J.~P. May.
\newblock {\em {Rings, modules and algebras in stable homotopy theory. With an
  appendix by M. Cole}}.
\newblock Mathematical Surveys and Monographs 47. American Mathematical
  Society, Providence, RI, 1997.

\bibitem[GM97]{GM:MU-modules}
J.~P.~C. Greenlees and J.~P. May.
\newblock {Localization and completion theorems for MU-module spectra}.
\newblock {\em Ann. of Math. (2) 146}, (3):509–544, 1997.

\bibitem[Gro]{grodal:burnside}
J.~Grodal.
\newblock {The Burnside ring of the p-completed classifying space of a finite
  group}.
\newblock {In preparation.}

\bibitem[GW18]{GW}
J.~J. Guti\'{e}rrez and D.~White.
\newblock Encoding equivariant commutativity via operads.
\newblock {\em Algebr. Geom. Topol.}, 18(5):2919--2962, 2018.

\bibitem[HH]{HH:EqvarSymMon}
M.~A. Hill and M.~J. Hopkins.
\newblock {Equivariant symmetric monoidal structures}.
\newblock {\texttt{arXiv:1610.03114}}.

\bibitem[HH14]{HH:EqvarMultClosure}
M.~A. Hill and M.~J. Hopkins.
\newblock {Equivariant multiplicative closure}.
\newblock {\em {Algebraic topology: applications and new directions, Contemp.
  Math.}}, 620, 2014.
\newblock {Amer. Math. Soc., Providence, RI}.

\bibitem[HHR16]{HHR}
M.~A. Hill, M.~J. Hopkins, and D.~C. Ravenel.
\newblock {On the nonexistence of elements of Kervaire invariant one}.
\newblock {\em {Ann. of Math. (2)}}, 184(1):1--262, 2016.

\bibitem[LMS86]{LMS}
L.~G. {Lewis, Jr.}, J.~P. May, and M.~Steinberger.
\newblock {\em {Equivariant stable homotopy theory. With contributions by J. E.
  McClure.}}
\newblock {Springer-Verlag, Berlin}, 1986.

\bibitem[Nak12]{nakaoka:ideals}
H.~Nakaoka.
\newblock {Ideals of Tambara functors}.
\newblock {\em {Adv. Math.}}, 230(4-6):2295--2331, 2012.

\bibitem[Oda14]{oda}
F.~Oda.
\newblock {On multiplicative induction}.
\newblock {\em {RIMS Kokyuroku}}, 1872:151--157, 2014.

\bibitem[Rub]{rubin:realization}
J.~Rubin.
\newblock {On the realization problem for $N_\infty$ operads}.
\newblock \texttt{arXiv:1705.03585}.

\bibitem[Sch]{schwede:lecture_ESHT_2018}
S.~Schwede.
\newblock {Lectures on equivariant stable homotopy theory}.
\newblock {Lecture notes, July 19, 2018. Available on the author's website:
  \texttt{http://www.math.uni-bonn.de/people/schwede/}}.

\bibitem[Seg71]{segal:ESHT}
G.~Segal.
\newblock {Equivariant stable homotopy theory}.
\newblock {\em Actes, Congr{\'e}s Intern. Math., Tome 2 (1970)}, pages 59--63,
  Paris, 1971.

\bibitem[Str12]{strickland:tambara}
N.~P. Strickland.
\newblock {Tambara functors}.
\newblock 2012.
\newblock \texttt{arXiv:1205.2516}.

\bibitem[Tam93]{tambara}
D.~Tambara.
\newblock {On multiplicative transfer}.
\newblock {\em Comm. Algebra}, 21(4):1393--1420, 1993.

\bibitem[tD75]{tomDieck:orbittypenII}
T.~tom Dieck.
\newblock Orbittypen und \"{a}quivariante {H}omologie. {II}.
\newblock {\em Arch. Math. (Basel)}, 26(6):650--662, 1975.

\bibitem[tD78]{tomDieck:idempotent}
T.~tom Dieck.
\newblock {Idempotent elements in the Burnside ring}.
\newblock {\em J. Pure Appl. Algebra}, 10(3):239--247, 1977/78.

\bibitem[Whi]{white:monoidal_v2}
D.~White.
\newblock {Monoidal Bousfield Localization and Algebras over Operads}.
\newblock \texttt{arXiv:1404.5197v2}.

\end{thebibliography}

\end{document}